\documentclass[11pt,oneside,reqno]{amsart}
\usepackage{geometry}
\usepackage[utf8]{inputenc}
\usepackage{amstext,latexsym,amsbsy,amsmath,amssymb,amsthm,mathtools,relsize,geometry,enumerate}
\usepackage{hyperref}
\hypersetup{
  colorlinks   = true, 
  urlcolor     = blue, 
  linkcolor    = red, 
  citecolor   = red 
}
\usepackage{mathtools,enumerate,enumitem}

\usepackage{graphicx}
\usepackage{epstopdf}
\setkeys{Gin}{width=\linewidth,totalheight=\textheight,keepaspectratio}
\graphicspath{{./images/}}

\usepackage{multicol}
\usepackage{booktabs}
\usepackage{mathrsfs}
\usepackage{color}

\newcommand{\nbb}{\mathbb{N}}
\newcommand{\rbb}{\mathbb{R}}

\renewcommand{\L}{\mathcal{L}}

\newcommand{\W}{\mathcal{W}}
\newcommand{\Hcal}{\mathcal{H}}
\newcommand{\Pcal}{\mathcal{P}}
\newcommand{\B}{\mathcal{B}}
\newcommand{\la}{\langle}
\newcommand{\ra}{\rangle}

\renewcommand{\o}{\circ}
\newcommand{\ve}{\varepsilon}

\newcommand{\Psitwo}{\Phi}

\newcommand{\nt}{\notag}

\newcommand{\Ut}{\tilde{U}}

\newcommand{\ut}{\tilde{u}}
\newcommand{\vt}{\tilde{v}}

\newcommand{\ci}{\circ}

\newcommand{\vbar}{\overline{v}}

\newcommand{\ubar}{\overline{u}}

\newcommand{\dt}{\tilde{d}}

\newcommand{\dtbn}{\tilde{d}_{\beta,n}}

\newcommand{\mi}{\wedge}

\newcommand{\Tr}{\text{Tr}}

\renewcommand{\d}{\text{d}}

\newcommand{\domain}{\mathcal{O}}

\newcommand{\f}{\varphi}
\newcommand{\g}{\psi}

\newcommand{\grad}{\nabla}
\newcommand{\Dom}{\text{Dom}}

\newcommand{\ddt}{\tfrac{\d}{\d t}}

\newcommand{\E}{\mathbb{E}}

\renewcommand{\P}{\mathbb{P}}

\newcommand{\close}{\!\!\!}

\theoremstyle{plain}
\newtheorem{theorem}{Theorem}[section]

\newtheorem{lemma}[theorem]{Lemma}
\newtheorem{assumption}[theorem]{Assumption}
\newtheorem{proposition}[theorem]{Proposition}
\newtheorem{definition}[theorem]{Definition}

\newtheorem{remark}[theorem]{Remark}

\numberwithin{equation}{section}

\title{Polynomial mixing of a stochastic wave equation with dissipative damping}

\author{ Hung D.~Nguyen$^1$}

\address{$^1$ Department of Mathematics, University of Tennessee, Knoxville, Tennessee, USA}

\begin{document}

\begin{abstract}
We study the long time statistics of a class of semi--linear wave equations modeling the motions of a particle suspended in continuous media while being subjected to random perturbations via an additive Gaussian noise. By comparison with the nonlinear reaction settings, of which the solutions are known to possess geometric ergodicity, we find that, under the impact of nonlinear dissipative damping, the mixing rate is at least polynomial of any order. This relies on a combination of Lyapunov conditions, the contracting property of the Markov transition semigroup as well as the notion of $d$--small sets.
\end{abstract}

\maketitle

\section{Introduction} \label{sec:intro}

Let $\domain\subset \rbb^d$, $d\le 3$ be a bounded open domain with smooth boundary. We consider the following system in the unknown variable $u(t)=u(x,t):\domain\times[0,\infty)\to\rbb$
\begin{align} 
\partial_{tt} u(t)&=-\partial_t u(t)+\triangle u(t) - \f'(\partial_t u(t))+ Q\, \partial_t w(t),\label{eqn:wave:original}\\
u(t)\big|_{\partial\domain}&=0,\quad u(0)=u_0\in H^1_0(\domain),\quad \partial_t u(0)=v_0\in L^2(\domain), \nt
\end{align}
where $\f\in C^2(\rbb)$ is the nonlinearity satisfying certain polynomial growth conditions, $w(t)$ is a cylindrical Wiener process taking values in $L^2(\domain)$ and $Q:L^2(\domain)\to L^2(\domain)$ is a symmetric linear bounded map. By Newton's second law, equation~\eqref{eqn:wave:original} can be interpreted as the displacement field of a particle suspended in a randomly continuous medium, under the impact of the stochastic forcing term $Q\partial_t w$, the interactions with surrounding particles represented by the Laplacian, a constant damping force as well as a nonlinear damping force depending on the velocity \cite{barbu2007stochastic,kim2008stochastic,
pardoux1975equations}. 

In the absence of the nonlinear damping ($\f'\equiv 0$), large--time behaviors of equation~\eqref{eqn:wave:original}, including those settings with the appearance of a reaction term $f(u)$, are well studied \cite{barbu2002stochastic,cerrai2006smoluchowski,
cerrai2020convergence,martirosyan2014exponential,
nguyen2022small}. That is under suitable assumptions on the potentials, the following equation
\begin{align} 
\partial_{tt} u(t)&=-\partial_t u(t)+\triangle u(t) + f(u(t))+ Q\, \partial_t w(t),\label{eqn:wave:potential}
\end{align}
is known to admit a unique invariant measure which is exponentially attractive \cite{cerrai2006smoluchowski, martirosyan2014exponential,nguyen2022small}. On the other hand, although ergodicity of the full equation~\eqref{eqn:wave:original} was investigated in \cite{barbu2007stochastic,kim2008stochastic}, to the best of the author's knowledge, a mixing rate has not been addressed before. The goal of this paper is to make
progress on bridging this gap in the convergent speed toward equilibrium between~\eqref{eqn:wave:potential} with reaction and \eqref{eqn:wave:original} with nonlinear damping.

Historically, stochastic wave equations of the form~\eqref{eqn:wave:original} was studied as early as in the work of \cite{pardoux1975equations}. Under a broad class of the nonlinearity $\f$, the well--posedness was established in \cite{barbu2007stochastic} whereas finite--time explosions in the presence of multiplicative noise was proven in \cite{gao2013stochastic}. By employing the classical Krylov--Bogoliubov, one can construct invariant probability measures in $H^1_0(\domain)\times L^2(\domain)$ for~\eqref{eqn:wave:original} \cite{barbu2007stochastic, kim2008stochastic}. Unique ergodicity was proven in \cite{barbu2007stochastic} by showing that regardless of initial conditions, the solutions always converge to one another as time tends to infinity. In our work, we tackle the problem of mixing rate and prove the following result: given sufficient regular initial conditions, the dynamics is attractive toward the unique invariant probability measure at least polynomially fast in a suitable Wasserstein distance; see Theorem \ref{thm:polynomial-mixing} for a precise statement.

Despite of a rich literature on geometric ergodicity for stochastic partial differential equations (SPDEs) \cite{hairer2008spectral,hairer2011theory,
hairer2011asymptotic,martirosyan2014exponential,
mattingly2002exponential,
nguyen2022ergodicity,
nguyen2022small}, settings satisfying only subgeometric mixing rates seem to receive less attention, owing to the fact that the dynamics typically do not possess a strong dissipative mechanism. To mention a few examples, results in this direction were established for non--linear Schr\"odinger equations \cite{debussche2005ergodicity}, and complex Ginzburg--Landau equation \cite{nersesyan2008polynomial}. In both cases, Foias--Prodi type estimates together with coupling arguments are exploited to show that the solutions are polynomially attractive toward their unique invariant probability measures.  Traditionally, if the Markov transition process is strong Feller, one can employ the strategy of \emph{small} sets \cite{meyn2012markov} while making use of Lyapunov technique to measure the convergent rate in total variation metric. In our present work, since we do not expect~\eqref{eqn:wave:original} to satisfy such a strong Feller property, we draw upon the framework developed in \cite{hairer2011asymptotic} and later popularized in \cite{butkovsky2014subgeometric,butkovsky2020generalized,
durmus2016subgeometric,kulik2017ergodic,
kulik2015generalized} to overcome the difficulty. More specifically, the method is based on the concepts of \emph{d--small} sets and the \emph{contracting} property of the Markov kernels with respect to a suitably chosen distance $d$. Here, we note that the notion of \emph{d--small} sets is weaker than the mentioned \emph{small} sets \cite{butkovsky2014subgeometric,meyn2012markov}. Nevertheless, it turns out to be useful for the investigation of ergodicity of equation~\eqref{eqn:wave:original}. In addition to our work, this approach has found many applications in a variety of continuous time systems \cite{butkovsky2014subgeometric,
butkovsky2020generalized,glatt2021mixing,hairer2011asymptotic,
nguyen2022ergodicity}, as well as discrete time settings \cite{durmus2016subgeometric}. On the other hand, unlike the usual approach where the distance $d$ is bounded, the significant difference of our work from literature is the appearance of unbounded distances in the mixing results, which amounts to the lack of irreducibility. To circumvent the issue, we slightly modify the contracting argument in \cite{butkovsky2014subgeometric,kulik2017ergodic} tailored to~\eqref{eqn:wave:original}. Although the argument is specifically presented for stochastic wave equations, we hope the technique may be applicable for other systems where neither strong dissipation nor irreducibility are available.

Another important key ingredient in the mixing result is the Lyapunov functional that is needed to quantify the convergent rate. As demonstrated previously in \cite{barbu2007stochastic}, it is not known whether~\eqref{eqn:wave:original} possesses ``nice" exponential decaying estimates. Instead, given sufficient regular initial conditions $U_0$, we obtain a polynomial bound of the form
\begin{align} \label{ineq:Lyapunov}
P_tg(U_0)+\int_0^t P_s(g^\varepsilon)(U_0)\d s\le g(U_0)+t,\quad t\ge 0,
\end{align}
for a constant $\varepsilon\in(0,1)$, an appropriately chosen function $g$; see Lemma~\ref{lem:moment-bound:H^2} below. In \eqref{ineq:Lyapunov}, $P_t$ denotes the Markov transition semigroup associated with~\eqref{eqn:wave:original}. As a result, following the approach detailed in \cite{butkovsky2014subgeometric,kulik2017ergodic}, it can be shown that the returning time to $d$--small sets possesses polynomial moment bounds \cite{douc2009subgeometric,meyn2012markov}. This together with the fact that $P_t$ is non--expanding and bounded sets are $d$--small, we conclude the convergent rate is at least polynomial of any order; see Theorem~\ref{thm:polynomial-mixing}. It is worth to mention that the analysis in this note as well as in \cite{barbu2007stochastic,kim2008stochastic} are all valid thanks to the assumption that the nonlinear effect from $\f'$ is dominated by the constant damping force. As a trade--off on the nonlinear restrictions, we are able to obtain the Lyapunov estimate of~\eqref{ineq:Lyapunov}--typed, which is very convenient for the purpose of establishing the mixing results. Also, as a bypass product of~\eqref{ineq:Lyapunov}, we derive higher moment bounds of the invariant measure, thereby concluding the polynomial convergent rate of any power.

The rest of the paper is organized as follows: in Section~\ref{sec:result}, we introduce all the functional settings as well as the main assumptions on the nonlinear damping term and noise structure. We also state our main results in this section, including Proposition~\ref{prop:moment-bound} on the moment bounds of invariant measures and Theorem~\ref{thm:polynomial-mixing} on the polynomial ergodicity. In Section~\ref{sec:moment-bound}, we perform a priori energy estimates on the solutions that will be employed to prove the main results. In Section~\ref{sec:ergodicity}, we prove the main results while making use of the auxiliary estimates from Section~\ref{sec:moment-bound}.

\section{Assumptions and main results} \label{sec:result}

\subsection{Functional setting} \label{sec:functional-setting}
Given $\domain$ a smooth bounded domain in $\rbb^d$, let $L^p$, $1\le p\le \infty$, denote the usual space $L^p(\domain)$. In particular, when $p=2$, we denote by $H$ the Hilbert space $L^2(\domain)$ endowed with the inner product $\la\cdot,\cdot\ra_H$ and the induced norm $\|\cdot\|_H$.

Let $A$ be the realization of $-\triangle$ in $H$ endowed with the Dirichlet boundary condition and the domain $\Dom(A)=H^1_0(\domain)\cap H^2(\domain)$. It is well--known that there exists an orthonormal basis $\{e_k\}_{k\ge 1}$ in $H$ that diagonalizes $A$, i.e.,  
\begin{equation}\label{eqn:Ae_k=alpha_k.e_k}
Ae_k=\alpha_k e_k,
\end{equation}
for a sequence of positive numbers $\alpha_1\le \alpha_2\le\dots$ diverging to infinity. 

For each $r\in\rbb$, we denote
\begin{equation}
H^r=\Dom(A^{r/2}),
\end{equation}
endowed with the inner product
\begin{align*}
\la u_1,u_2\ra_{H^r}=\la A^{r/2}u_1,A^{r/2}u_2\ra_H.
\end{align*}
In view of~\eqref{eqn:Ae_k=alpha_k.e_k}, the inner product in $H^r$ may be rewritten as \cite{robinson2001infinite}
\begin{align*}
\la u_1,u_2\ra_{H^r}=\sum_{k\ge 1}\alpha_k^{r}\la u_1,e_k\ra_H\la u_2,e_k\ra_H.
\end{align*}
The induced norm in $H^r$ then is given by
\begin{align*}
\|u\|^2_{H^r}=\sum_{k\ge 1}\alpha_k^{r}|\la u_1,e_k\ra_H|^2.
\end{align*}
It is well--known that the embedding $H^{r_1}\subset H^{r_2}$ is compact for $r_1>r_2$.

For each $\beta\in\rbb$, let $\Hcal^\beta$ be the product space given by
\begin{equation} \label{form:H^beta.x.H^(beta-1)}
\Hcal^\beta = H^\beta \times H^{\beta-1},
\end{equation}
endowed by the norm
\begin{align*}
\|(u,v)\|^2_{\Hcal^\beta}=\|u\|^2_{H^\beta}+\|v\|^2_{H^{\beta-1}}.
\end{align*}
Let $\Pcal(\Hcal^\beta)$ be the space of Borel probability measures in $\Hcal^\beta$. Given a Borel measurable function $f:\Hcal^\beta\to [0,\infty)$, we denote $\Pcal_f(\Hcal^\beta)$ to be the subset of $\Pcal(\Hcal^\beta)$ consisting of elements that are integrable with respect to $f$. If $\nu\in \Pcal_f(\Hcal^\beta)$, we denote $\nu(f):=\int_{\Hcal^\beta}f(U)\nu(\d U)$.

\subsection{Assumptions} \label{sec:result:assumption}
In this subsection, we state the main assumptions on the nonlinearity and noise structure that will be employed throughout the rest of the paper.

Regarding the noise, we assume that $w(t)$ is a cylindrical Wiener process on $H$, whose decomposition is given by
$$w(t)=\sum_{k\ge 1}e_kB_k(t),$$
where $\{e_k\}_{k\ge 1}$ is the orthonormal basis of $H$ as in \eqref{eqn:Ae_k=alpha_k.e_k} and $\{B_k(t)\}_{k\ge 1}$ is a sequence of independent standard one--dimensional Brownian motions, each defined on the same stochastic basis $\mathcal{S}=(\Omega, \mathcal{F},\{\mathcal{F}_t\}_{t\ge 0},\P)$ \cite{karatzas2012brownian}. Concerning the linear operator $Q$, we impose the following nominal assumption \cite{barbu2007stochastic,barbu2002stochastic, bonaccorsi2012asymptotic,cerrai2020convergence,
da2014stochastic}:
\begin{assumption}\label{cond:Q} 
$Q:H\to H$ is a symmetric, non--negative, bounded linear map. We assume that $Q$ is diagonalizable by the orthonormal basis $\{e_k\}_{k\ge 1}$ given by \eqref{eqn:Ae_k=alpha_k.e_k}, i.e., there exists a non--negative sequence $\{\lambda_k\}_{k\ge 1}$ such that
\[
Qe_k=\lambda_k e_k,\qquad k\ge 1,\]
and that
\begin{equation} \label{cond:Q:H^4}
\emph{Tr}(QAQ)=\sum_{k\ge 1} \lambda_k^2\alpha_k<\infty,
\end{equation}
where $\{\alpha_k\}_{k\ge 1}$ are eigenvalues of $A$ as in \eqref{eqn:Ae_k=alpha_k.e_k}.

\end{assumption}

Here, we recall the notion $\Tr(f)=\sum_{k\ge 1}\la fe_k,e_k\ra_H$ for $f:H\to H$. We also remark that all of the constants $\lambda_k$ need not be strictly positive, as long as they satisfy the condition \eqref{cond:Q:H^4}. Concerning the nonlinearity $\f:\rbb\to\rbb$, we impose the following conditions: \cite{barbu2007stochastic,kim2008stochastic}

\begin{assumption} \label{cond:phi} $\f\in C^2$ satisfies $\f'(0)=0$.

\begin{enumerate}[noitemsep,topsep=0pt,wide=\parindent, label=\arabic*.,ref=\theassumption.\arabic*]

\item There exist positive constants $a_1,a_2,a_3$ and $\lambda$ such that for all $x\in\rbb$,
\begin{equation} \label{cond:phi:phi'(x)=O(x^lambda)} 
|\f'(x)|\le a_1(1+|x|^{\lambda}),
\end{equation}
and
\begin{equation} \label{cond:phi:x.phi'(x)<-x^(lambda+1)}
x\f'(x)\ge a_2|x|^{\lambda+1}-a_3.
\end{equation}

\item There exist a positive constant $a_4$ and a real constant $a_5$ such that the second derivative $\f''$ satisfies 
\begin{equation}  \label{cond:phi:phi''=O(x^(lambda-1))}
|\f''(x)| \le a_4(|x|^{\lambda-1}+1),\quad x\in\rbb,
\end{equation}
and
\begin{equation} \label{cond:phi:inf.phi''>-1}
\inf_{x\in\rbb}\f''(x)=a_5>-1.
\end{equation}
\item The constant $\lambda$ satisfies the following condition depending on the dimension $d$ 
\begin{align} \label{cond:lambda}
\lambda\in [1,3] \quad \text{if}\,\,\, d=1,\quad \lambda\in [1,3) \quad \text{if}\,\,\,  d=2,\quad\lambda\in [1,2] \quad \text{if}\,\,\,  d=3.
\end{align}

\end{enumerate}
\end{assumption}

\subsection{Main results} \label{sec:result:main-result}

Following \cite{barbu2007stochastic,barbu2002stochastic} and setting $v(t)=\partial_t u(t)$, we may recast~\eqref{eqn:wave:original} as
\begin{equation} \label{eqn:wave}
\begin{aligned}
\d u(t)&=v(t)\d t,\\
\d v(t)&=-Au(t)\d t-v(t)\d t-\f'(u(t))\d t+ Q\d w(t),\\
(u(0),v(0))&=(u_0,v_0)\in \Hcal^1.
\end{aligned}
\end{equation}
The generator associated with~\eqref{eqn:wave} is denoted by \cite{cerrai2020convergence}
\begin{equation} \label{form:L}
\L g(u,v)= \la D_u g, v\ra_H+\la D_v g,-Au-v(t)-\f'(u(t))\ra_H+ \Tr(D_{vv}gQQ^*),
\end{equation}
and is defined for all $g \in C^2(\Hcal^1;\rbb)$ such that $\Tr(D_{vv}gQQ^*)<\infty$.

Under the assumptions imposed in Section~\ref{sec:result:assumption}, it can be shown that~\eqref{eqn:wave} admits a unique mild solution $U(t)=(u(t),v(t))$ in the sense of \cite[Definition 2.1]{barbu2007stochastic}. The argument of the well--posedness can be derived via either the classical existence results for second order stochastic equations of hyperbolic type \cite[Theorem 2.1]{pardoux1975equations}, Yosida approximation \cite[Theorem 2.3]{barbu2007stochastic}, or Galerkin approximation \cite{kim2008stochastic}. As a consequence of the well--posedness, we can thus introduce the Markov transition probabilities of the solution $U(t)$ by
\begin{align*}
P_t(U_0,A):=\P(U(t;U_0)\in A),
\end{align*}
which are well--defined for $t\ge 0$, initial states $U_0\in\Hcal^1$ and Borel sets $A\subseteq \Hcal^1$. Letting $\B_b(\Hcal^1)$ denote the set of bounded Borel measurable functions $f:\Hcal^1 \rightarrow \rbb$, the associated Markov semigroup $P_t:\B_b(\Hcal^1)\to\B_b(\Hcal^1)$ is defined and denoted by
\begin{align}\label{form:P_t^m}
P_t f(U_0)=\E[f(U(t;U_0))], \quad f\in \B_b(\Hcal^1).
\end{align}
Let $\Pcal (\Hcal^1)$ be the space of probability measures in $\Hcal^1$. The push--forward of $\nu\in\Pcal (\Hcal^1)$ under the action of $P_t$ is denoted by $P_t\nu$ and defined as
\begin{align*}
P_t\nu (A)= \int_{\Hcal^1}P_t(U,A)\nu(\d U).
\end{align*}

We now turn to the topic of ergodicity for~\eqref{eqn:wave}. Recall that a probability measure $\nu\in \Pcal (\Hcal^1)$ is said to be {\it\textbf{invariant}} for the semigroup $P_t$ if 
\begin{align*}
P_t\nu = \nu.
\end{align*}
It is well--known that under Assumption~\ref{cond:Q} and Assumption~\ref{cond:phi}, $P_t$ admits a unique invariant probability measure $\nu$. The existence of $\nu$ was obtained via the classical Krylov--Bogoliubov tightness argument \cite{barbu2007stochastic,kim2008stochastic} applied to a sequence of time--averaged measures. On the other hand, the uniqueness of $\nu$ was established in~\cite[Theorem 4.1]{barbu2007stochastic} by showing that two solutions always converge to the same limit, regardless of initial conditions.

In Proposition~\ref{prop:moment-bound} below, we assert that the unique invariant probability measure $\nu$ satisfies moment bound of any order in higher regularity. 

\begin{proposition} \label{prop:moment-bound}
Let $\nu$ be the unique invariant measure of~\eqref{eqn:wave}. Then, for all $p>0$,
\begin{align} \label{ineq:moment-bound:nu}
\int_{\Hcal^1}\big( \|u\|_{H^2}+\|v\|_{H^1}+\|v\|^{\lambda+1}_{L^{\lambda+1}}\big)^p\nu(\emph{d}u,\emph{d}v)<\infty.
\end{align}
\end{proposition}

The moment bound of $\nu$ will be obtained via a priori energy estimates of the solution $U(t)$ in $\Hcal^2$, presented in Lemma~\ref{lem:moment-bound:H^2}. Later in Section~\ref{sec:ergodicity:proof-of-theorem}, the proof of Proposition~\ref{prop:moment-bound} will be supplied in detail. Also, the moment bounds on $\nu$ allows us to establish the polynomial convergence rate of any order, which we describe next.

With regard to the mixing rate of~\eqref{eqn:wave}, as mentioned in the introduction, we will draw upon the framework developed in \cite{hairer2006ergodicity,hairer2008spectral} and later popularized in \cite{butkovsky2014subgeometric,butkovsky2020generalized,
hairer2011theory,hairer2011asymptotic,kulik2017ergodic,kulik2015generalized}, tailored to our settings. Since the analysis involves Wasserstein distances, we briefly review related notions below.

Recall that a function $d:\Hcal^1\times\Hcal^1\to [0,\infty)$ is called \emph{distance--like} if it is symmetric, lower semi--continuous and $d(U,\Ut)=0$ if and only if $U=\Ut$; see \cite[Definition 4.3]{hairer2011asymptotic}. Let $\W_d$ be the Wasserstein distance in $\Pcal (\Hcal^1)$ associated with $d$ and given by
\begin{align} \label{form:W_d}
\W_d(\nu_1,\nu_2)=\inf \E\, d(X,Y),
\end{align}
where the infimum is taken over all bivariate random variables $(X,Y)$ such that $X\sim \nu_1$ and $Y\sim \nu_2$. On the one hand, we remark that since $d$ is only distance--like, $\W_d$ need not satisfy the usual triangle inequality. Nevertheless, $\W_d$ provides a reasonable way to measure distance between probability measures in the sense that $\W_d(\nu_1,\nu_2)=0$ if and only if $\nu_1=\nu_2$ \cite{butkovsky2020generalized, hairer2011asymptotic}. On the other hand, in case $d$ is indeed a metric, by the dual Kantorovich Theorem, $\W_{d}$ is equivalently defined as \cite[Theorem 5.10]{villani2008optimal}
\begin{align} \label{form:W_d:dual-Kantorovich}
\W_{d}(\nu_1,\nu_2)=\sup_{[f]_{\text{Lip},d}\leq 1}\Big|\int_{\Hcal^1}f(U)\nu_1(\d U)-\int_{\Hcal^1}f(U)\nu_2(\d U)\Big|,
\end{align}
where
\begin{align} \label{form:Lipschitz}
[f]_{\text{Lip},d}=\sup_{U\neq \Ut}\frac{|f(U)-f(\tilde{U})|}{d(U,\tilde{U})}.
\end{align}
To measure the convergence of $P_t$ toward $\nu$, we introduce the distance-like function
\begin{align} \label{form:d}
d(U,\Ut)=\|U-\Ut\|_{\Hcal^1}^2,\quad U,\Ut\in\Hcal^1.
\end{align}

As mentioned in Section~\ref{sec:intro}, since the dynamics is not strongly dissipative in $\Hcal^1$ \cite{barbu2007stochastic}, we are only able to obtain Lyapunov functionals in higher regularity. As a consequence, given any initial data in $\Hcal^2$, the solutions of \eqref{eqn:wave} are polynomially attractive toward $\nu$. This is summarized in the main result below:

\begin{theorem} \label{thm:polynomial-mixing}
Let $U\in\Hcal^2$ be given, $\nu$ be the unique invariant measure of~\eqref{eqn:wave} and $\lambda$ be the constant as in Assumption~\ref{cond:phi}. Then, there exists $T^*>0$ such that for all $n\ge 4$, $\gamma\in (0,1/\lambda)$, the following holds
\begin{align} \label{ineq:polynomial-mixing}
\W_d\big(P_t(U,\cdot),\nu\big) \le  C\big(1+\|U\|^{2n}_{\Hcal^2}\big)t^{-\frac{3(n-1+\gamma)}{4(1-\gamma)}},\quad t\ge T^*, \, U\in\Hcal^2,
\end{align}
where $\W_d$ is the Wasserstein distance associated with $d$ defined in~\eqref{form:d}, and $C>0$ is a positive constant independent of $U$ and $t$.
\end{theorem}

Following the framework in \cite{butkovsky2014subgeometric,butkovsky2020generalized,
hairer2011asymptotic}, the argument of the proof of~\eqref{ineq:polynomial-mixing} combines the contracting property of $\W_d$ where $d$ is defined above in~\eqref{form:d}, the existences of Lyapunov functions as well as ``good" \emph{d--small} sets in $\Hcal^2$. However, we note that we are not able to directly employ \cite[Theorem 2.4]{butkovsky2014subgeometric} since the distance $d$ is unbounded. As a result, it is necessary to modify the argument found in \cite{butkovsky2014subgeometric} adapted to our setting. In Section~\ref{sec:ergodicity:proof-of-theorem}, we will supply the terminologies of contracting and $d-$small sets, as well as provide the proof of Theorem~\ref{thm:polynomial-mixing} in detail. 

It is also worth to point out that~\eqref{ineq:polynomial-mixing} alone is not sufficient to guarantee the existence and uniqueness of invariant probability measures in $\Pcal(\Hcal^1)$. Nevertheless, unique ergodicity follows from the estimates in Section~\ref{sec:moment-bound}, and was previously established in \cite{barbu2007stochastic,kim2008stochastic}.

\section{A priori moment estimates} \label{sec:moment-bound}

Throughout the rest of the paper, $c$ and $C$ denote generic positive constants that may change from line to line. The main parameters that they depend on will appear between parenthesis, e.g., $c(T,q)$ is a function of $T$ and $q$.

In this section, we collect useful moment estimates on the solutions of~\eqref{eqn:wave}. In Lemma~\ref{lem:moment-bound:int_0^t|v|}, stated and proven next, we assert an exponential moment bounds of $(u(t),v(t))$ in $\Hcal^1$. Lemma~\ref{lem:moment-bound:int_0^t|v|} will be particularly employed in the proof of Lemma~\ref{lem:moment-bound:ubar} below. In turn, Lemma~\ref{lem:moment-bound:ubar} will be invoked to prove the main results.

\begin{lemma} \label{lem:moment-bound:int_0^t|v|}
Given $(u_0,v_0)\in \Hcal^1$, let $(u(t),v(t))$ be the solution of~\eqref{eqn:wave} with initial condition $(u_0,v_0)$. Then, the following holds for all $\beta>0$ sufficiently small and $t\ge 0$ independent of $(u_0,v_0)$
\begin{align} \label{ineq:moment-bound:H^1:exponential}
&\E\exp \Big\{  \sup_{r\in[0,t]}\beta\big(\|u(r)\|^2_{H^1}+\|v(r)\|^2_H\big)+\beta \int_0^t \|v(r)\|^2_H+2a_2\|v(r)\|^{\lambda+1}_{L^{
\lambda+1}}\emph{d} r \Big\}   \nt \\
&\le \big( 1+4\beta\emph{Tr}(QQ^*)\big)\exp\Big\{\beta\big(\|u_0\|^2_{H^1}+\|v_0\|^2_H\big)+\beta\big(\emph{Tr}(QQ^*)+2a_3|\domain|\big)t\Big\}.
\end{align}
In the above, $a_2,a_3,\lambda$ are as in~\eqref{cond:phi:x.phi'(x)<-x^(lambda+1)} and $|\domain|$ denotes the volume of $\domain$.
\end{lemma}
\begin{proof}
From~\eqref{eqn:wave}, by It\^o's formula, we have
\begin{align*}
&\d\big( \|u(t)\|^2_{H^1}+\|v(t)\|^2_H\big) \\
&= -2\|v(t)\|^2_H\d t-2\la \f'(v(t)), v(t)\ra_H\d t+2\la v(t),Q\d w(t)\ra_H+\Tr(QQ^*)\d t.
\end{align*}
In view of condition~\eqref{cond:phi:x.phi'(x)<-x^(lambda+1)}, 
\begin{align*}
-2\la \f'(v(t)), v(t)\ra_H \le -2a_2\|v(t)\|^{\lambda+1}_{L^{
\lambda+1}}+2a_3|\domain|,
\end{align*}
where $a_2,\,a_3$ are as in~\eqref{cond:phi:x.phi'(x)<-x^(lambda+1)}. It follows that
\begin{align*}
&\d \big(\|u(t)\|^2_{H^1}+\|v(t)\|^2_H\big) \\
&\le -2\|v(t)\|^2_H\d t-2a_2\|v(t)\|^{\lambda+1}_{L^{
\lambda+1}}\d t+2\la v(t),Q\d w(t)\ra_H+\big(\Tr(QQ^*)+2a_3|\domain|\big)\d t.
\end{align*}
In particular, for all $\beta>0$, it holds that
\begin{align*}
&\beta\big(\|u(t)\|^2_{H^1}+\|v(t)\|^2_H\big)-\beta\big(\|u_0\|^2_{H^1}+\|v_0\|^2_H\big)\\
&\le -\beta \int_0^t 2\|v(r)\|^2_H+2a_2\|v(r)\|^{\lambda+1}_{L^{
\lambda+1}}\d r+M(t)+\beta\big(\Tr(QQ^*)+2a_3|\domain|\big)t,
\end{align*}
where the Martingale term $M(t)$ is defined as
\begin{align*}
M(t)=\int_0^t 2\beta\la v(r),Q\d w(r)\ra_H.
\end{align*}
Turning to~\eqref{ineq:moment-bound:H^1:exponential}, we aim to employ the exponential Martingale inequality applied to $M(t)$. To see this, note that the quadratic variation process $\la M\ra(t)$ satisfies
\begin{align*}
\d \la M\ra(t) =4\beta^2\|Qv(t)\|^2_H\d t\le 4\beta^2\Tr(QQ^*)\|v(t)\|^2_H\d t,
\end{align*}
whence
\begin{align*}
-\beta\int_0^t \|v(r)\|^2_H\d r= -\tfrac{1}{2}\cdot \tfrac{1}{2\beta\Tr(QQ^*)} \int_0^t 4\beta^2\Tr(QQ^*)\|v(r)\|^2_H\d r\le -\tfrac{1}{2}\cdot \tfrac{1}{2\beta\Tr(QQ^*)}\la M\ra(t).
\end{align*}
As a consequence, we obtain
\begin{align*}
&\beta\big(\|u(t)\|^2_{H^1}+\|v(t)\|^2_H\big)-\beta\big(\|u_0\|^2_{H^1}+\|v_0\|^2_H\big)\\
&\le -\beta \int_0^t \|v(r)\|^2_H+2a_2\|v(r)\|^{\lambda+1}_{L^{
\lambda+1}}\d r+\beta\big(\Tr(QQ^*)+2a_3|\domain|\big)t\\
&\qquad +M(t)-\tfrac{1}{2}\cdot \tfrac{1}{2\beta\Tr(QQ^*)}\la M\ra(t).
\end{align*}
Now applying the exponential Martingale inequality to $M(t)$ gives
\begin{align*}
\P\Big(\sup_{t\ge 0}\Big[ M(t)- \tfrac{1}{2}\cdot \tfrac{1}{2\beta\Tr(QQ^*)}\la M\ra(t)  \Big]\ge R    \Big)\le e^{-  \frac{1}{2\beta\Tr(QQ^*)}R}.
\end{align*}
Hence, for $\beta$ sufficiently small, we obtain
\begin{align*}
&\E\exp\Big\{\sup_{t\ge 0}\Big[ M(t)- \tfrac{1}{2}\cdot \tfrac{1}{2\beta\Tr(QQ^*)}\la M\ra(t)  \Big]    \Big\}\\
&\le 1+\int_1^\infty \P\Big(\sup_{t\ge 0}\Big[ M(t)- \tfrac{1}{2}\cdot \tfrac{1}{2\beta\Tr(QQ^*)}\la M\ra(t)  \Big]\ge \log R    \Big)\d R\\
&\le 1+\int_1^\infty R^{-  \frac{1}{2\beta\Tr(QQ^*)}}\d R= 1+ \frac{1}{  \frac{1}{2\beta\Tr(QQ^*)}-1}\le 1+4\beta\Tr(QQ^*).
\end{align*}
Altogether, we arrive at the bound for all $\beta$ sufficiently small and $t\ge 0$,
\begin{align*}&
\E\exp \Big\{  \sup_{r\in[0,t]}\beta\big(\|u(r)\|^2_{H^1}+\|v(r)\|^2_H\big)+\beta \int_0^t \|v(r)\|^2_H+2a_2\|v(r)\|^{\lambda+1}_{L^{
\lambda+1}}\d r \Big\}\\
&\le \E\exp\Big\{\sup_{t\ge 0}\Big[ M(t)- \tfrac{1}{2}\cdot \tfrac{1}{2\beta\Tr(QQ^*)}\la M\ra(t)  \Big]    \Big\}\\
&\qquad\qquad\times\exp\Big\{\beta\big(\|u_0\|^2_{H^1}+\|v_0\|^2_H\big)+\beta\big(\Tr(QQ^*)+2a_3|\domain|\big)t\Big\}\\
&\le \big( 1+4\beta\Tr(QQ^*)\big)\exp\Big\{\beta\big(\|u_0\|^2_{H^1}+\|v_0\|^2_H\big)+\beta\big(\Tr(QQ^*)+2a_3|\domain|\big)t\Big\}.
\end{align*}
This establishes~\eqref{ineq:moment-bound:H^1:exponential}, as claimed.
\end{proof}

Next, in Lemma~\ref{lem:moment-bound:ubar}, we establish moment bounds on the difference $\|U(t;U_0)-U(t;\Ut_0)\|_{\Hcal^1}$ while specifically making use of condition \eqref{cond:lambda} of $\lambda$. In particular, the results of Lemma~\ref{lem:moment-bound:ubar} will be employed in the proof of Proposition~\ref{prop:ergodicity} in Section~\ref{sec:ergodicity:auxiliary}.

\begin{lemma} \label{lem:moment-bound:ubar}
Given $(u_0,v_0), (\ut_0,\vt_0)\in \Hcal^1$, let $(u(t),v(t))$ and $(\ut(t),\vt(t))$, respectively, be the solutions of~\eqref{eqn:wave} with initial conditions $(u_0,v_0)$ and $(\ut_0,\vt_0)$. Then, the following holds:

1. For all $t\ge 0$,
\begin{align} \label{ineq:moment-bound:|ubar|^2<|u_0|^2}
\|u(t)-\ut(t)\|^2_{H^1}+\|v(t)-\vt(t)\|^2_{H} \le \|u_0-\ut_0\|^2_{H^1}+\|v_0-\vt_0\|^2_{H}.
\end{align}

2. For all $\varepsilon>0$ sufficiently small,
\begin{align} \label{ineq:moment-bound:|ubar|:exponential-decay}
&\E\Big[\|u(t)-\ut(t)\|^2_{H^1}+\|v(t)-\vt(t)\|^2_{H}+\varepsilon\la u(t)-\ut(t),v(t)-\vt(t)\ra_H  \Big] \nt \\
&\le \big(\|u_0-\ut_0\|^2_{H^1}+\|v_0-\vt_0\|^2_{H}+\varepsilon\la u_0-\ut_0,v_0-\vt_0\ra_H \big) \nt \\
&\qquad \times \big( 1+C\varepsilon^2\big)\exp\Big\{\varepsilon^2C\big(\|(u_0,v_0)\|^2_{\Hcal^1}+\|(\ut_0,\vt_0)\|^2_{\Hcal^1}\big)- \tfrac{1}{8}\ve t\Big\},\quad t\ge 0.
\end{align}
In the above, $C$ is a positive constant independent of $\ve$, $t$, $(u_0,v_0)$, and $(\ut_0,\vt_0)$.

\end{lemma}

It is worth to mention that the factors of $\ve^2$ as well as $-\ve t$ on the right--hand side of~\eqref{ineq:moment-bound:|ubar|:exponential-decay} are crucial for the mixing results. To be more specific, we will ultimately choose $\ve$ sufficiently small in order to establish the existence of $d$--small sets. See the proofs of Proposition~\ref{prop:ergodicity} and Lemma~\ref{lem:contracting:W_d} for a further discussion of this point.

\begin{proof}[Proof of Lemma~\ref{lem:moment-bound:ubar}]
1. Setting $\ubar(t)=u(t)-\ut(t)$ and $\vbar(t)=v(t)-\vt(t)$, observe that $(\ubar(t),\vbar(t))$ satisfies the equation
\begin{align*}
\ddt \ubar(t)&=\vbar(t),\quad \ddt \vbar(t)=-A\ubar(t)-\vbar(t)-[\f'(v(t))-\f'(\vt(t)) ],\\
\ubar(0)&=\ubar_0=u_0-\ut_0,\quad \vbar(t)=\vbar_0=v_0-\vt_0.
\end{align*}
A routine calculation while making use of~\eqref{cond:phi:inf.phi''>-1} gives
\begin{align}
\ddt \big(\|\ubar(t)\|^2_{H^1}+\|\vbar(t)\|^2_{H}\big)&=-2\|\vbar(t)\|^2_{H}-2\la \f'(v(t))-\f'(\vt(t)),\vbar(t)\ra_H   \nt \\
&\le -2(1+a_5)\|\vbar(t)\|^2_{H}, \label{ineq:d.|ubar|^2}
\end{align}
where $a_5$ is as in~\eqref{cond:phi:inf.phi''>-1}. Estimate~\eqref{ineq:moment-bound:|ubar|^2<|u_0|^2} now follows by integrating both sides of~\eqref{ineq:d.|ubar|^2} with respect to time.

2. With regard to \eqref{ineq:moment-bound:|ubar|:exponential-decay}, we first note that
\begin{align} \label{eqn:d<ubar,vbar>}
\ddt \la \ubar(t),\vbar(t)\ra_H &= \|\vbar(t)\|_H^2-\|\ubar(t)\|^2_{H^1}-\la \ubar(t),\vbar(t)\ra_H-\la \f'(v(t))-\f'(\vt(t)),\ubar(t)\ra_H .
\end{align}
To estimate the last term on the above right--hand side, we follow along the lines of the proof of \cite[Theorem 4.1]{barbu2007stochastic} while making use of Assumption~\ref{cond:phi}, part 3 on $\lambda$. By Holder's inequality, we have that
\begin{align*}
\la \f'(v(t))-\f'(\vt(t)),\ubar(t)\ra_H & \le \|\ubar(t)\|_{H^1}\|\f'(v(t))-\f'(\vt(t)\|_{H^{-1}}.
\end{align*}
Using Sobolev embedding, we further estimate
\begin{align*}
\|\f'(v(t))-\f'(\vt(t)\|_{H^{-1}}&\le c \|\f'(v(t))-\f'(\vt(t)\|_{L^{q^*}}\\
&\le c\Big(\int_{\domain}|\vbar(t)|^{q^*}\big( |v(t)|^{\lambda-1}+|\vt(t)|^{\lambda-1}+1  \big)^{q^*}\d x\Big)^{1/q^*}.
\end{align*}
In the above, $q^*$ is such that $\frac{1}{q^*}=1-\frac{1}{p^*}$ where $p^*=\infty$ if $d=1$, $p^*\ge 2$ if $d=2$ and $p^*=6$ if $d=3$ \cite{barbu2007stochastic}. Also, the choice $\lambda$ as in~\eqref{cond:lambda} implies
\begin{align*}
\frac{q^*}{2}+\frac{(\lambda-1)q^*}{\lambda+1}\le 1.
\end{align*}
So, we invoke Holder's inequality again to obtain
\begin{align*}
\Big(\int_{\domain}|\vbar(t)|^{q^*}\big( |v(t)|^{\lambda-1}+|\vt(t)|^{\lambda-1}+1  \big)^{q^*}\d x\Big)^{1/q^*}&\le c\|\vbar(t)\|_H \big( \|v(t)\|^{\lambda-1}_{L^{\lambda+1}}+\|\vt(t)\|^{\lambda-1}_{L^{\lambda+1}}+1\big),
\end{align*}
whence
\begin{align*}
\la \f'(v(t))-\f'(\vt(t)),\ubar(t)\ra_H &\le \|\ubar(t)\|_{H^1}\|\f'(v(t))-\f'(\vt(t)\|_{H^{-1}}\\
&\le c\|\ubar(t)\|_{H^1}\|\vbar(t)\|_H \big( \|v(t)\|^{\lambda-1}_{L^{\lambda+1}}+\|\vt(t)\|^{\lambda-1}_{L^{\lambda+1}}+1\big).
\end{align*}
This together with~\eqref{eqn:d<ubar,vbar>} yields the estimate
\begin{align} \label{ineq:d<ubar,vbar>}
\ddt \la \ubar(t),\vbar(t)\ra_H &\le \|\vbar(t)\|_H^2-\|\ubar(t)\|^2_{H^1}-\la \ubar(t),\vbar(t)\ra_H \nt\\
&\qquad+c\|\ubar(t)\|_{H^1}\|\vbar(t)\|_H \big( \|v(t)\|^{\lambda-1}_{L^{\lambda+1}}+\|\vt(t)\|^{\lambda-1}_{L^{\lambda+1}}+1\big).
\end{align}

Turning back to~\eqref{ineq:moment-bound:|ubar|:exponential-decay}, let $\ve>0$ be given and be chosen later. We combine~\eqref{ineq:d<ubar,vbar>} with~\eqref{ineq:d.|ubar|^2} to obtain
\begin{align*}
&\ddt\big(\|\ubar(t)\|^2_{H^1}+\|\vbar(t)\|^2_{H}+\ve \la \ubar(t),\vbar(t)\ra_H \big)\\
&\le -\big[2(1+a_5)-\ve\big]\|\vbar(t)\|^2_{H}-\ve\|\ubar(t)\|^2_{H^1}-\ve\la \ubar(t),\vbar(t)\ra_H \nt\\
&\qquad+\ve c\|\ubar(t)\|_{H^1}\|\vbar(t)\|_H \big( \|v(t)\|^{\lambda-1}_{L^{\lambda+1}}+\|\vt(t)\|^{\lambda-1}_{L^{\lambda+1}}+1\big).
\end{align*}
Since $\ve$ is assumed to be sufficiently small and $C$ does not depend on $\ve$, we infer
\begin{align*}
& -\big[2(1+a_5)-\ve\big]\|\vbar(t)\|^2_{H}-\ve\|\ubar(t)\|^2_{H^1}-\ve\la \ubar(t),\vbar(t)\ra_H \nt+\ve c\|\ubar(t)\|_{H^1}\|\vbar(t)\|_H\\
&  \le -(1+a_5)\|\vbar(t)\|^2_{H}-\tfrac{1}{2}\ve\|\ubar(t)\|^2_{H^1} .
\end{align*}
Also, (recalling $\lambda\le 3$)
\begin{align*}
&\ve c\|\ubar(t)\|_{H^1}\|\vbar(t)\|_H \big( \|v(t)\|^{\lambda-1}_{L^{\lambda+1}}+\|\vt(t)\|^{\lambda-1}_{L^{\lambda+1}}\big)\\
&\le \varepsilon^2  c\|\ubar(t)\|_{H^1}^2 \big( \|v(t)\|^{2(\lambda-1)}_{L^{\lambda+1}}+\|\vt(t)\|^{2(\lambda-1)}_{L^{\lambda+1}}\big) + \tfrac{1}{2}(1+a_5)\|\vbar(t)\|_H^2\\
&\le \varepsilon^2  c\|\ubar(t)\|_{H^1}^2 \big( \|v(t)\|^{\lambda+1}_{L^{\lambda+1}}+\|\vt(t)\|^{\lambda+1}_{L^{\lambda+1}}+1\big) + \tfrac{1}{2}(1+a_5)\|\vbar(t)\|_H^2.
\end{align*}
It follows that by taking $\varepsilon$ small enough
\begin{align*}
&\ddt\big(\|\ubar(t)\|^2_{H^1}+\|\vbar(t)\|^2_{H}+\ve \la \ubar(t),\vbar(t)\ra_H \big)\\
& \le -\tfrac{1}{2}(1+a_5)\|\vbar(t)\|^2_{H}-\tfrac{1}{2}\ve\|\ubar(t)\|^2_{H^1} + \varepsilon^2  c\|\ubar(t)\|_{H^1}^2 \big( \|v(t)\|^{\lambda+1}_{L^{\lambda+1}}+\|\vt(t)\|^{\lambda+1}_{L^{\lambda+1}}+1\big)\\
&\le \Big[-\tfrac{1}{4}\varepsilon+\varepsilon^2 c\big( \|v(t)\|^{\lambda+1}_{L^{\lambda+1}}+\|\vt(t)\|^{\lambda+1}_{L^{\lambda+1}}\big)\Big] \big(\|\ubar(t)\|^2_{H^1}+\|\vbar(t)\|^2_{H}+\ve \la \ubar(t),\vbar(t)\ra_H \big),
\end{align*}
implying
\begin{align*}
&\|\ubar(t)\|^2_{H^1}+\|\vbar(t)\|^2_{H}+\ve \la \ubar(t),\vbar(t)\ra_H\\
&\le \big(\|\ubar_0\|^2_{H^1}+\|\vbar_0\|^2_{H}+\ve \la \ubar_0,\vbar_0\ra_H\big) \exp\Big\{ -\tfrac{1}{4}\varepsilon t+\varepsilon^2 c\int_0^t \big( \|v(r)\|^{\lambda+1}_{L^{\lambda+1}}+\|\vt(r)\|^{\lambda+1}_{L^{\lambda+1}}\big)\d r \Big\} .
\end{align*}
In view of~\eqref{ineq:moment-bound:H^1:exponential}, the following holds
\begin{align*}
&\E \exp\Big\{\varepsilon^2 c\int_0^t \big( \|v(r)\|^{\lambda+1}_{L^{\lambda+1}}+\|\vt(r)\|^{\lambda+1}_{L^{\lambda+1}}\big)\d r \Big\}\\
&\le  \big( 1+C\varepsilon^2\big)\exp\Big\{\varepsilon^2C\big(\|(u_0,v_0)\|^2_{\Hcal^1}+\|(\ut_0,\vt_0)\|^2_{\Hcal^1}\big)+\ve^2 C t\Big\}.
\end{align*}
As a consequence, we obtain the bound in expectation for all $\ve$ sufficiently small
\begin{align*}
&\E\Big[\|\ubar(t)\|^2_{H^1}+\|\vbar(t)\|^2_{H}+\ve \la \ubar(t),\vbar(t)\ra_H\Big]\\
&\le \big(\|\ubar_0\|^2_{H^1}+\|\vbar_0\|^2_{H}+\ve \la \ubar_0,\vbar_0\ra_H\big)\big( 1+C\varepsilon^2\big)\exp\Big\{\varepsilon^2C\big(\|(u_0,v_0)\|^2_{\Hcal^1}+\|(\ut_0,\vt_0)\|^2_{\Hcal^1}\big)- \tfrac{1}{8}\ve t\Big\}.
\end{align*}
We emphasize that $C$ is a positive constant independent of $\ve$. This establishes~\eqref{ineq:moment-bound:|ubar|:exponential-decay}, thereby finishing the proof.

\end{proof}

Lastly, we establish a moment bound in $\Hcal^2$ provided the initial data $(u_0,v_0)\in\Hcal^2$. For this purpose, we introduce the function
\begin{align} \label{form:Phi}
\Psitwo(u,v) = \|u\|^2_{H^2}+\|v\|^2_{H^1}+\la u,v\ra_{H^1} +\tfrac{1}{2}\|u\|^2_{H^1}+\|\f(v)\|_{L^1}.
\end{align}

\begin{lemma} \label{lem:moment-bound:H^2}
Given $(u_0,v_0)\in \Hcal^2$, let $(u(t),v(t))$ be the solution of~\eqref{eqn:wave} with initial condition $(u_0,v_0)$. Then, for all $n\ge1$ and $\gamma\in(0,1/\lambda)$ where $\lambda$ is as in Assumption~\ref{cond:phi}, the following holds
\begin{align} \label{ineq:moment-bound:H^2}
\E\big[ \Psitwo(u(t),v(t))^n\big]+c_{n,\gamma}\int_0^t \E\big[\Psitwo(u(s),v(s))^{n-1+\gamma}\big]\emph{d}s \le \Psitwo(u_0,v_0)^n +C_{n,\gamma}t,\quad t\ge 0.
\end{align}
In the above, $\Psitwo$ is the function defined in~\eqref{form:Phi}, $c_{n,\gamma}$ and $C_{n,\gamma}$ are positive constants independent of $(u_0,v_0)$ and $t$.
\end{lemma}
\begin{proof}
We proceed to prove~\eqref{ineq:moment-bound:H^2} by induction on $n$. We start with the base case $n=1$. From equation~\eqref{eqn:wave}, a routine calculation gives
\begin{align*} 
\L \|\f(v)\|_{L^1}= -\la Au,\f'(v)\ra_{H}-\la v,\f'(v)\ra_H-\|\f'(v)\|^2_H+\tfrac{1}{2}\Tr(\f''(v) QQ^* ).
\end{align*}
Also,
\begin{align*}
&\L \Big(\|u\|^2_{H^2}+\|v\|^2_{H^1}+\la u,v\ra_{H^1} +\tfrac{1}{2}\|u\|^2_{H^1} \Big)\\
&= -2\|v\|^2_{H^1}-2\la \f''(v)\grad v,\grad v\ra_{H}+\Tr(QAQ^*)-\|u\|^2_{H^2}-\la Au,\f'(v)\ra_H.
\end{align*}
Combining the above identities with $\Psitwo$ as in~\eqref{form:Phi}, we obtain
\begin{align} \label{eqn:L.Psi_2}
\L \Psitwo(u,v)&= -2\|v\|^2_{H^1}-2\la \f''(v)\grad v,\grad v\ra_{H}+\Tr(QAQ^*) \nt \\
&\qquad-\|Au+\f'(v)\|^2_H-\la v,\f'(v)\ra_H+\tfrac{1}{2}\Tr(\f''(v) QQ^* ).
\end{align}
In view of condition \eqref{cond:phi:x.phi'(x)<-x^(lambda+1)}, we readily have
\begin{align*}
-\la v,\f'(v)\ra_H \le -a_2\|v\|^{\lambda+1}_{L^{\lambda+1}}+a_3|\domain|,
\end{align*}
where $a_2,a_3,\lambda$ are as in Assumption~\ref{cond:phi} and $|\domain|$ denotes the volume of $\domain$. Also, from~\eqref{cond:phi:inf.phi''>-1}, $\inf_{x\in\rbb} \f''(x)=a_5>-1$ implies
\begin{align*}
-2\|v\|^2_{H^1}-2\la \f''(v)\grad v,\grad v\ra_{H} \le -2(1+a_5)\|v\|^2_{H^1}.
\end{align*}
With regards to the trace terms on the right--hand side of~\eqref{eqn:L.Psi_2}, we invoke~\eqref{cond:Q:H^4} to see that
\begin{align*}
\Tr(QAQ^*) =\sum_{k\ge 1}\lambda_k^2\alpha_k<\infty.
\end{align*}
Also, together with~\eqref{cond:phi:phi''=O(x^(lambda-1))}, we estimate
\begin{align*}
\Tr(\f''(v) QQ^* ) &=\sum_{k\ge 1} \lambda_k^2\int_{\domain}\f''(v)|e_k|^2\d x\\
&\le \sum_{k\ge 1} \lambda_k^2\|e_k\|_{L^\infty}^2\big(a_4\|v\|^{\lambda-1}_{L^{\lambda-1}}+a_4|\domain|\big).
\end{align*}
Since $H^2\hookrightarrow L^\infty$, we further deduce
\begin{align*}
\sum_{k\ge 1} \lambda_k^2\|e_k\|_{L^\infty}^2 \le c\sum_{k\ge 1} \lambda_k^2\|Ae_k\|_{H}^2=c\sum_{k\ge 1} \lambda_k^2\alpha_k^2<\infty,
\end{align*}
whence
\begin{align*}
\tfrac{1}{2}\Tr(\f''(v) QQ^* )& \le c\sum_{k\ge 1} \lambda_k^2\alpha_k^2\big(a_4\|v\|^{\lambda-1}_{L^{\lambda-1}}+a_4|\domain|\big)\le  \tfrac{1}{2}a_2\|v\|^{\lambda+1}_{L^{\lambda+1}}+C.
\end{align*}
As a consequence, we obtain
\begin{align} \label{ineq:L.Psi_2}
\L\Psitwo(u,v)&\le -2(1+a_5)\|v\|^2_{H^1}-\tfrac{1}{2}a_2\|v\|^{\lambda+1}_{L^{\lambda+1}}-\|Au+\f'(v)\|^2_H+C.
\end{align}
Now, fixing a $\gamma\in(0,1)$ to be chosen later, we note that
\begin{align*}
\|Au\|_H^{2\gamma}\le c\|Au+\f'(v)\|^{2\gamma}_H+c\|\f'(v)\|^{2\gamma}_H.
\end{align*}
Recalling condition~\eqref{cond:phi:phi'(x)=O(x^lambda)}, since $\lambda\le 3$, we use Holder's inequality together with the fact that $H^1\hookrightarrow L^6$ to estimate
\begin{align*}
\|\f'(v)\|^{2\gamma}_H \le c\|v\|^{2\gamma\lambda}_{L^{2\lambda}}+c&\le c\|v\|^{2\gamma\lambda}_{L^6}+c\le c\|v\|^{2\gamma\lambda}_{H^1}+c.
\end{align*}
So, 
\begin{align*}
\|Au\|_H^{2\gamma}\le c\big(\|Au+\f'(v)\|^{2\gamma}_H+\|v\|^{2\gamma\lambda}_{H^1}+1\big).
\end{align*}
Choosing $\gamma<1/\lambda$, it follows that
\begin{align*}
\|Au\|_H^{2\gamma} \le \tfrac{1}{2}\|Au+\f'(v)\|^2_H+(1+a_5)\|v\|^{2}_{H^1}+C.
\end{align*}
From~\eqref{ineq:L.Psi_2}, we obtain
\begin{align*}
\L\Psitwo(u,v) \le -(1+a_5)\|v\|^2_{H^1}-\tfrac{1}{2}a_2\|v\|^{\lambda+1}_{L^{\lambda+1}}-\tfrac{1}{2}\|Au+\f'(v)\|^2_H-\|Au\|^{2\gamma}_H+C.
\end{align*}
Since $U$ is dominated by $|x|^{\lambda+1}$, we further deduce
\begin{align} \label{ineq:L.Psi_2:b}
\L\Psitwo(u,v) \le -c\Psitwo(u,v)^\gamma+C.
\end{align}
By It\^o's formula, this implies estimate~\eqref{ineq:moment-bound:H^2} for the base case $n=1$.

Now, assuming~\eqref{ineq:moment-bound:H^2} holds for up to $n-1$, consider the case $n\ge 2$. We compute
\begin{align} \label{eqn:L.Psi_2^n}
\L \Psitwo^n = n\Psitwo^{n-1}\L\Psitwo+\tfrac{1}{2}
n(n-1)\Psitwo^{n-2}\sum_{k\ge 1}\lambda_k^2\big|\la 2Av+Au+\f'(v),e_k\ra_H\big|^2.
\end{align}
With regard to the first term on the above right hand side, we invoke~\eqref{ineq:L.Psi_2} and~\eqref{ineq:L.Psi_2:b} to estimate
\begin{align*}
n\Psitwo^{n-1}\L\Psitwo& \le -c\Psitwo^{n-1}\big(\|v\|^2_{H^1}+\|Au+\f'(v)\|^2_H\big)\\
&\qquad -c\Psitwo^{n-1+\gamma}+C\Psitwo^{n-1}.
\end{align*}
Reasoning as in the base case $n=1$, we also have
\begin{align*}
\Psitwo\ge  c\big(\|u\|^2_{H^2}+\|v\|^2_{H^1}\big)  
&\ge c\big(\|u\|^2_{H^2}+\|v\|^2_{H^1}+\|\f'(v)\|^{2\gamma}_H\big)-C \\
&\ge c \big(\|v\|^2_{H^1}+\|Au+\f'(v)\|^2_H\big)^{\gamma}-C.
\end{align*}
It follows that
\begin{align*}
n\Psitwo^{n-1}\L\Psitwo& \le -c\Psitwo^{n-2}\big(\|v\|^2_{H^1}+\|Au+\f'(v)\|^2_H\big)^{1+\gamma}\\
&\qquad -c\Psitwo^{n-1+\gamma}+C\Psitwo^{n-1}+C\Psitwo^{n-2},
\end{align*}
whence
\begin{align} \label{ineq:L.Psitwo^n:a}
n\Psitwo^{n-1}\L\Psitwo& \le -c\Psitwo^{n-2}\big(\|v\|^2_{H^1}+\|Au+\f'(v)\|^2_H\big)^{1+\gamma}-c\Psitwo^{n-1+\gamma}+C.
\end{align}
With regard to the second term on the right--hand side of~\eqref{eqn:L.Psi_2^n}, we invoke Cauchy--Schwarz inequality to obtain
\begin{align}
\tfrac{1}{2}
&n(n-1)\Psitwo^{n-2}\sum_{k\ge 1}\lambda_k^2\big|\la 2Av+Au+\f'(v),e_k\ra_H\big|^2 \nt \\
&\le n(n-1)\Psitwo^{n-2}\Tr(QQ^*)\big(4\|v\|^2_{H^1}+\|Au+\f'(v)\|^2_H\big) \nt \\
&\le \tfrac{1}{2}c\Psitwo^{n-2}\big(\|v\|^2_{H^1}+\|Au+\f'(v)\|^2_H\big)^{1+\gamma}+C\Psitwo^{n-2}.\label{ineq:L.Psitwo^n:b}
\end{align}
We now combine~\eqref{eqn:L.Psi_2^n} together with~\eqref{ineq:L.Psitwo^n:a} and \eqref{ineq:L.Psitwo^n:b} to deduce
\begin{align*}
\L \Psitwo^n  \le-c\Psitwo^{n-1+\gamma}+C.
\end{align*}
By It\^o's formula, the above produces~\eqref{ineq:moment-bound:H^2} for the general case $n\ge2$, thereby concluding the proof.

\end{proof}

\section{Ergodicity of \eqref{eqn:wave}}  \label{sec:ergodicity}

\subsection{Proof of the main results} \label{sec:ergodicity:proof-of-theorem} In this subsection, we prove the main results while making use of the moment estimates established in Section~\ref{sec:moment-bound}. 

We start with Proposition~\ref{prop:moment-bound}, which is consequence of Lemma~\ref{lem:moment-bound:H^2}. The proof of Proposition~\ref{prop:moment-bound} is quite standard and can be found in literature. Since the argument is short, we include it here for the sake of completeness. 

\begin{proof}[Proof of Proposition \ref{prop:moment-bound}]
Let $\nu_t\in \Pcal(\Hcal^1)$ be the time--average measure defined as
\begin{align*}
\nu_t(\cdot)=\frac{1}{t}\int_0^t P_s(0, \cdot )\d s.
\end{align*}
By the Krylov--Bogoliubov procedure as well as the uniqueness of $\nu$, it is known that $\nu_t$ converges weakly to $\nu$. With regard to the moment bound~\eqref{ineq:moment-bound:nu}, for $R>0$, we invoke~\eqref{ineq:moment-bound:H^2} to obtain the bound
\begin{align*}
\frac{1}{t}\int_0^t \E \big[\Psitwo(u(s),v(s))^{n-1+\gamma}\mi R\big]\d s\le C, \quad t>0,
\end{align*}
for some positive constant $C$ independent of $R$. Since the above left--hand side converges to $\nu(\Psitwo^{n-1+\gamma}\mi R)$ as $t\to\infty$, we deduce
\begin{align*}
\nu(\Psitwo^{n-1+\gamma}\mi R) \le C.
\end{align*}
By virtue of the Monotone Convergence Theorem, we obtain $$\nu(\Psitwo^{n-1+\gamma})<\infty,$$
implying 
\begin{align*}
\int_{\Hcal^1} \big(\|u\|^2_{H^2}+\|v\|^2_{H^1}+\|v\|^{\lambda+1}_{L^{\lambda+1}}\big)^n\nu(\d u,\d v)<\infty,
\end{align*}
thanks to the fact that $\Psitwo(u,v)$ dominates $\|u\|^2_{H^2}+\|v\|^2_{H^1}+\|v\|^{\lambda+1}_{L^{\lambda+1}}$. Since this holds for all $n\in\nbb$, we conclude~\eqref{ineq:moment-bound:nu}, as claimed.
\end{proof}

We now turn to Theorem~\ref{thm:polynomial-mixing}. For the reader's convenience, we recall the notions of \emph{contracting} and \emph{d--small} sets below:
\begin{definition} \label{defn:contracting-d-small}
1. A Wasserstein distance $\W_d$ is called contracting for $P_t$ if the following holds \cite{butkovsky2014subgeometric,butkovsky2020generalized,kulik2015generalized}
\begin{align*}
\W_d\big(P_t(U,\cdot),P_t(\Ut,\cdot)\big) \le d(U,\Ut),\quad t\ge 0,
\end{align*}
for all $U,\Ut\in\Hcal^1$.

2. A set $A\in\Hcal^1$ is called $d$-- \emph{small} for $P_t$ if there exists $\rho=\rho(t,A)\in(0,1)$ such that \cite[Definition 2.2]{butkovsky2014subgeometric} 
\begin{align*}
\W_d\big(P_t(U,\cdot),P_t(\Ut,\cdot)\big) \le (1-\rho)d(U,\Ut), \quad U,\Ut\in A.
\end{align*}
\end{definition}

\begin{remark} \label{rem:contracting} We note that the notion of contracting in Definition~\ref{defn:contracting-d-small}, part 1, is similar to \cite[Condition (2.2)]{butkovsky2014subgeometric}. In particular, it is weaker than the usual notion of contracting found in \cite[Definition 2.1]{butkovsky2020generalized} and \cite[Definition 4.6]{hairer2011asymptotic}, where it is required that
\begin{align*}
\W_d\big(P_t(U,\cdot),P_t(\Ut,\cdot)\big) \le \alpha d(U,\Ut),
\end{align*}
for $\alpha<1$ whenever $d(U,\Ut)<1$.
\end{remark}

In order to actual measure the convergence rate via Lyapunov functions, we introduce the function
\begin{align} \label{form:psi_n}
\g_n(x)=|x|^{\frac{n-1+\gamma}{n}},\quad x\in\rbb,
\end{align}
defined for $n\ge1$ and $\gamma\in(0,1/\lambda)$ as in~\eqref{ineq:moment-bound:H^2}. Three of the main ingredients to prove Theorem~\ref{thm:polynomial-mixing} are given below in Proposition~\ref{prop:ergodicity} whose proof is deferred to Section~\ref{sec:ergodicity:auxiliary}.

\begin{proposition}  \label{prop:ergodicity}
1. (Lyapunov functions) Let $\Psitwo$ be defined in~\eqref{form:Phi}. Then,
\begin{align} \label{ineq:Lyapunov:Phi_n}
P_t \Psitwo^n(U) \le \Psitwo(U)^n -c\int_0^t P_s (\g_n\ci\Psitwo^n)(U)\emph{d}s+Ct, \quad U\in\Hcal^2,\, t\ge 0,
\end{align}
where $\g_n$ is as in~\eqref{form:psi_n}, $c$ and $C$ are positive constants independent of $U$ and $t$.

2. (Contracting property) For all $t\ge 0$, 
\begin{align} \label{ineq:contracting:W_d(U_0,Utilde_0)<d(U_0,Utilde_0)}
\W_{d}\big( P_t(U_0,\cdot),P_t(\Ut_0,\cdot) \big) \le d(U_0,\Ut_0),\quad U_0,\Ut_0\in\Hcal^1,
\end{align}
where $d(U_0,\Ut_0)=\|U_0-\Ut_0\|_{\Hcal^1}^2$ as in~\eqref{form:d}, and $\W_d$ is the Wasserstein distance associated with $d$.

3. (d--small sets) There exists $t_0>0$ such that for all $R>0$, $t\ge t_0$,
\begin{align} \label{ineq:contracting:W_d(U_0,Utilde_0)<(1-rho)d(U_0,Utilde_0)}
\W_{d}\big( P_t(U_0,\cdot),P_t(\Ut_0,\cdot) \big) \le (1-\rho)d(U_0,\Ut_0),
\end{align}
holds for a positive constant $\rho=\rho(R,t)$ and for all $U_0,\Ut_0\in B_{n}^R $ where
\begin{align} \label{form:B_n^R}
B^R_{n}=\{U\in\Hcal^2:\Phi(U)^n\le R\}.
\end{align} 

\end{proposition}

\begin{remark}
We note that while the contracting property~\eqref{ineq:contracting:W_d(U_0,Utilde_0)<d(U_0,Utilde_0)} holds in $\Hcal^1$, the Lyapunov estimate~\eqref{ineq:Lyapunov:Phi_n} as well as the existence of $d$--small sets only hold in $\Hcal^2$, hence the convergent rate result, cf. Theorem~\ref{thm:polynomial-mixing}, for every initial data $U\in\Hcal^2$.
\end{remark}

Assuming the results of Proposition~\ref{prop:ergodicity}, we turn to the polynomial mixing stated in Theorem~\ref{thm:polynomial-mixing}. As mentioned in Section~\ref{sec:result:main-result}, we are not able to directly employ \cite[Theorem 2.4]{butkovsky2014subgeometric} owing to the fact that $d(U,\Ut)=\|U-\Ut\|^2_{\Hcal^1}$ is unbounded. It is thus necessary to modify the argument from \cite{butkovsky2014subgeometric} tailored to our settings. More specifically, Theorem~\ref{thm:polynomial-mixing} is established via a series of lemmas below, of which, the first result (in Lemma~\ref{lem:Lyapunov:W}) relies on Lyapunov structure~\eqref{ineq:Lyapunov:Phi_n}. The proof of Lemma~\ref{lem:Lyapunov:W} can be found in \cite[Theorem 2.4]{butkovsky2014subgeometric} and thus omitted.

\begin{lemma} \label{lem:Lyapunov:W}
Let $\g_n$ and $\Phi$ respectively be as in~\eqref{form:psi_n} and \eqref{form:Phi}. Then, there exists a function $W_n:\Hcal^2\to [1,\infty)$ such that
\begin{align} \label{cond:W_n-vs-Phi_n}
\g_n( \Phi(U)^n) \le  W_n(U) \le C\Phi(U)^n+C,\quad U\in\Hcal^2,
\end{align}
and
\begin{align} \label{cond:W_n}
P_{t_0}W_n(U)\le W_n(U)-\g_n(K_1 W_n(U))+K_2 \quad U\in\Hcal^2.
\end{align}
In the above, $t_0$ is the time constant as in Proposition~\ref{prop:ergodicity}, part 3, $C$, $K_1$ and $K_2$ are positive constants independent of $U$.
\end{lemma}

\begin{remark} \label{rem:W_n}
Following closely the proof of \cite[Theorem 2.4]{butkovsky2014subgeometric}, $W_n$ is actually constructed via the returning time to $B_n^R$ defined in \eqref{form:B_n^R}. More specifically, given $U_0\in\Hcal^2$, denote
\begin{align*}
\sigma_{n,R} =\inf\{m\in\nbb:U(mt_0;U_0)\in B^R_n\}.
\end{align*}
As a consequence of the energy estimate in Lemma \ref{lem:moment-bound:H^2}, by employing an argument similar to the proof of \cite[Lemma 4.7]{butkovsky2014subgeometric}, it can be shown that $\sigma_{n,R}$ is finite a.s. See also \cite[Proposition 22]{fort2005subgeometric}. Then, there exists $R=R(t_0)$ sufficiently large such that
\begin{align*}
W_n(U_0)=\E \Big[\sum_{i=0}^{\sigma_{n,R}}\Phi\big(U(it_0;U_0)\big)^n\Big],
\end{align*}
satisfies~\eqref{cond:W_n-vs-Phi_n} and \eqref{cond:W_n}. To ensure $W_n\in[1,\infty)$, we simply replace $W_n$ by $W_n+1$.
\end{remark}

Next, we establish a contracting property for $P_{t_0}$ making use of Lemma~\ref{lem:Lyapunov:W}. For this purpose, letting $W_n$ be the function as in Lemma~\ref{lem:Lyapunov:W}, we introduce the distance--like function $\dtbn:\Hcal^2\times\Hcal^2\to[0,\infty)$ for $\beta>0$ and $n\in\nbb$ given by
\begin{align} \label{form:dtilde_beta}
\dtbn(U,\Ut) = \sqrt{ d(U,\Ut)\big[1+\beta \g_n\big(W_n(U)+W_n(\Ut)  \big)   \big]},\quad U,\Ut\in\Hcal^2,
\end{align}
where $\g_n$ is defined in~\eqref{form:psi_n}.

\begin{lemma} \label{lem:W_dtilde:contracting:t_0}
Let $\dtbn$ be defined in~\eqref{form:dtilde_beta}, $t_0$ be the time constant as in Proposition~\ref{prop:ergodicity}, part 3, and $W_n$ be the function from Lemma~\ref{lem:Lyapunov:W}. Then, for all $n\ge 4$, there exists $\beta, c^*_n$ and $C^*_n>0$ such that the following holds
\begin{align} \label{ineq:W_dtilde:contracting:t_0}
&\W_{\dtbn}\big(P_{t_0}\nu_1,P_{t_0}\nu_2\big)  \nt \\
&\le 
\Big[ 1-c_n^*\g_n'\Big(\g_n^{-1}\Big( C_n^* \big[\nu_1( \g_n\o W_n)+\nu_2(\g_n\o W_n)\big]^{\frac{4}{3}} \big[\W_{\dtbn}(\nu_1,\nu_2)\big]^{-\frac{4}{3}} \Big) \Big)\Big]
\W_{\dtbn}\big(\nu_1,\nu_2\big),
\end{align}
for all $\nu_1\neq\nu_2\in \Pcal_{\g_n\ci \Phi^n}(\Hcal^2)$ where $\g_n$ and $\Phi$ are respectively defined in~\eqref{form:psi_n} and \eqref{form:Phi}.
\end{lemma}
\begin{proof}
We follow along the lines of \cite[Lemma 4.3]{butkovsky2014subgeometric} tailored to our settings where $d(U,\Ut)=\|U-\Ut\|_{\Hcal^1}^2$ is unbounded.

By convexity of $\W_{\dtbn}$, cf. \cite[Theorem 4.8]{villani2008optimal}, it holds that
\begin{align} \label{ineq:W_dtilde:contracting:t_0:a}
\W_{\dtbn}\big(P_{t_0}\nu_1,P_{t_0}\nu_2\big) \le\int_{\Hcal^2\times\Hcal^2}  \close \W_{\dtbn}\big(P_{t_0}(U_0,\cdot),P_{t_0}(\Ut_0,\cdot)\big) \pi(\d U_0,\d \Ut_0),
\end{align} 
for any coupling $\pi$ of $(\nu_1,\nu_2)$. We proceed to compare $\W_{\dtbn}\big(P_{t_0}(U_0,\cdot),P_{t_0}(\Ut_0,\cdot)\big)$ with $\dtbn(U_0,\Ut_0)$. Recalling $\dtbn$ from~\eqref{form:dtilde_beta} and letting $(X,Y)$ be a coupling of $\big(P_{t_0}(U_0,\cdot),P_{t_0}(\Ut_0,\cdot)\big)$, we apply Cauchy--Schwarz inequality together with the expression~\eqref{form:W_d} and the concavity of $\g_n$ to see that
\begin{align*}
&\W_{\dtbn}\big(P_{t_0}(U_0,\cdot),P_{t_0}(\Ut_0,\cdot)\big)\le   \E\,\dt_{\beta}(X,Y)\\
&\le \sqrt{\E d(X,Y)}\cdot \sqrt{1+\beta\E\g_n(W_n(X)+W_n(Y))}\\
&\le  \sqrt{\E d(X,Y)}\cdot \sqrt{1+\beta\g_n(\E W_n(X)+\E W_n(Y))}\\
&= \sqrt{\E d(X,Y)}\cdot \sqrt{1+\beta\g_n(P_{t_0} W_n(U_0)+P_{t_0} W_n(\Ut_0))}.
\end{align*}
Since the last implication above holds for all pairs of random variables $(X,Y)$, we invoke~\eqref{form:W_d} again to obtain
\begin{align}
&\W_{\dtbn}\big(P_{t_0}(U_0,\cdot),P_{t_0}(\Ut_0,\cdot)\big) \nt \\
&\le \sqrt{\W_{d}\big(P_{t_0}(U_0,\cdot),P_{t_0}(\Ut_0,\cdot)\big) }\cdot \sqrt{1+\beta\g_n(P_{t_0} W(U_0)+P_{t_0} W(\Ut_0))}.  \label{ineq:W_dtilde:contracting:t_0:b}
\end{align}
Now, let $N>R$ be two large constants and to be chosen later. There are three cases to be considered, depending on the value of $W_n(U_0)+W_n(\Ut_0)$. 

\emph{Case 1}: $W_n(U_0)+W_n(\Ut_0)\le R$. In this case, we invoke~\eqref{cond:W_n-vs-Phi_n} to see that $\max\{\Phi(U_0)^n,\Phi(\Ut_0)^n\}\le \g_n^{-1}(\frac{1}{c_n}R) $. In view of Proposition~\ref{prop:ergodicity}, part 3, cf.~\eqref{ineq:contracting:W_d(U_0,Utilde_0)<(1-rho)d(U_0,Utilde_0)}, there exists $\rho_1=\rho_1(t_0,R)\in(0,1)$ such that
\begin{align*}
\W_{d}\big( P_t(U_0,\cdot),P_t(\Ut_0,\cdot) \big) \le (1-\rho_1)d(U_0,\Ut_0).
\end{align*}
Also, from~\eqref{cond:W_n}, we have
\begin{align*}
P_{t_0} W_n(U_0)+P_{t_0} W_n(\Ut_0) \le W_n(U_0)+W_n(\Ut_0)+2K_2\le R+2K_2. 
\end{align*}
Combined the above estimates with~\eqref{ineq:W_dtilde:contracting:t_0:b} yields
\begin{align*}
\W_{\dtbn}\big(P_{t_0}(U_0,\cdot),P_{t_0}(\Ut_0,\cdot)\big) 
&\le \sqrt{(1-\rho_1)d(U_0,\Ut_0)}\cdot \sqrt{1+\beta\g_n(R+2K_2)}.
\end{align*}
Setting
\begin{align*}
\beta = \frac{\rho_1}{\g_n(R+2K_2)} ,
\end{align*}
we infer (recalling $\dtbn$ in~\eqref{form:dtilde_beta})
\begin{align}
\W_{\dtbn}\big(P_{t_0}(U_0,\cdot),P_{t_0}(\Ut_0,\cdot)\big) 
\le \sqrt{(1-\rho^2_1)d(U_0,\Ut_0)}&\le (1-\tfrac{1}{2}\rho^2_2)\sqrt{d(U_0,\Ut_0)}   \nt \\
&\le (1-\tfrac{1}{2}\rho^2_2)\dtbn(U_0,\Ut_0). \label{ineq:W_dtilde:contracting:t_0:case1}
\end{align}

\emph{Case 2}: $R<W_n(U_0)+W_n(\Ut_0)\le N$. In this case, we once again invoke~\eqref{cond:W_n} and the concavity of $\g_n$ to see that
\begin{align}
&\g_n\big(P_{t_0} W(U_0)+P_{t_0} W(\Ut_0)\big)  \nt \\
  &\le \g_n\big( W_n(U_0)+W_n(\Ut_0)-\g_n\o K_1W_n(U_0)-\g_n\o K_1W_n(\Ut_0)+2K_2  \big)\nt \\
&\le \g_n\big( W_n(U_0)+W_n(\Ut_0)-\g_n( K_1W_n(U_0)+ K_1W_n(\Ut_0))+2K_2  \big). \label{ineq:W_dtilde:contracting:t_0:c}
\end{align}
Picking $R$ sufficiently large enough such that $\g_n(R)>2K_2/\g_n(K_1)$, observe that for $x\in(R,N]$
\begin{align*}
\g_n( x-\g_n(K_1x)+2K_2  )-\g_n(x) \le -\g'(N)(\g_n(K_1x)-2K_2),
\end{align*}
implying (recalling $\g_n$ from~\eqref{form:psi_n})
\begin{align*}
\g_n( x-\g_n(K_1x)+2K_2  ) &\le \g_n(x)\Big( 1-\big[\tfrac{\g_n(K_1x)}{\g_n(x)}-\tfrac{2K_2}{\g_n(x)}\big]\g_n'(N) \Big)\\
&= \g_n(x)\Big( 1-\big[\g_n(K_1)-\tfrac{2K_2}{\g_n(x)}\big]\g_n'(N) \Big)\\
&\le \g_n(x)\Big( 1-\big[\g_n(K_1)-\tfrac{2K_2}{\g_n(R)}\big]\g_n'(N) \Big).
\end{align*}
Since $\lim_{N\to 0}\g_n'(N)=0$, for all $N$ sufficiently large, 
\begin{align*}
 1-\rho_2\g_n'(N):=1-\big[\g_n(K_1)-\tfrac{2K_2}{\g_n(R)}\big]\g_n'(N) \in (0,1).
\end{align*}
From~\eqref{ineq:W_dtilde:contracting:t_0:b} together with~\eqref{ineq:contracting:W_d(U_0,Utilde_0)<d(U_0,Utilde_0)}, we deduce
\begin{align*}
&\W_{\dtbn}\big(P_{t_0}(U_0,\cdot),P_{t_0}(\Ut_0,\cdot)\big) \nt \\
&\le \sqrt{\W_{d}\big(P_{t_0}(U_0,\cdot),P_{t_0}(\Ut_0,\cdot)\big) }\cdot \sqrt{1+\beta\g_n(P_{t_0} W(U_0)+P_{t_0} W(\Ut_0))}\\
&\le \sqrt{d(U_0,\Ut_0) }\cdot \sqrt{1+\beta\g_n(W(U_0)+ W(\Ut_0)) (1-\rho_2\g_n'(N)) }\\
&\le \sqrt{1-\rho_2\g_n'(N)}\sqrt{d(U_0,\Ut_0) }\cdot \sqrt{1+\beta\g_n(W(U_0)+ W(\Ut_0))  }.
\end{align*}
It follows that (recalling $\dtbn$ in~\eqref{form:dtilde_beta})
\begin{align}\label{ineq:W_dtilde:contracting:t_0:case2}
\W_{\dtbn}\big(P_{t_0}(U_0,\cdot),P_{t_0}(\Ut_0,\cdot)\big)\le (1-\tfrac{1}{2}\rho_2\g_n'(N))\dtbn(U_0,\Ut_0).
\end{align}

\emph{Case 3}: $W_n(U_0)+W_n(\Ut_0)> N$. In this case, we combine~\eqref{ineq:W_dtilde:contracting:t_0:b} with~\eqref{ineq:W_dtilde:contracting:t_0:c} and~\eqref{ineq:contracting:W_d(U_0,Utilde_0)<d(U_0,Utilde_0)} to infer
\begin{align*}
&\W_{\dtbn}\big(P_{t_0}(U_0,\cdot),P_{t_0}(\Ut_0,\cdot)\big) \\
&\le \sqrt{\W_{d}\big(P_{t_0}(U_0,\cdot),P_{t_0}(\Ut_0,\cdot)\big) }\cdot \sqrt{1+\beta\g_n(P_{t_0} W(U_0)+P_{t_0} W(\Ut_0))}\\
&\le \sqrt{d(U_0,\Ut_0) }\cdot \sqrt{1+\beta\g_n\big( W_n(U_0)+W_n(\Ut_0)-\g_n(K_1 W_n(U_0)+ K_1W_n(\Ut_0))+2K_2\big) }.
\end{align*}
Since $\g_n(N)>\g_n(R)>2K_2/\g_n(K_1)$, we deduce
\begin{align}\label{ineq:W_dtilde:contracting:t_0:case3}
&\W_{\dtbn}\big(P_{t_0}(U_0,\cdot),P_{t_0}(\Ut_0,\cdot)\big) \nt \\
&\le \sqrt{d(U_0,\Ut_0) }\cdot \sqrt{1+\beta\g_n\big( W_n(U_0)+W_n(\Ut_0)\big) }=\dtbn(U_0,\Ut_0).
\end{align}

Turning back to~\eqref{ineq:W_dtilde:contracting:t_0:a}, we collect~\eqref{ineq:W_dtilde:contracting:t_0:case1}, \eqref{ineq:W_dtilde:contracting:t_0:case2} and \eqref{ineq:W_dtilde:contracting:t_0:case3} to infer
\begin{align}
\W_{\dtbn}\big(P_{t_0}\nu_1,P_{t_0}\nu_2\big)
& \le\int_{\Hcal^2\times\Hcal^2}  \close \W_{\dtbn}\big(P_{t_0}(U_0,\cdot),P_{t_0}(\Ut_0,\cdot)\big) \pi(\d U_0,\d \Ut_0)     \nt \\
&\le \big(1-\tfrac{1}{2}\rho_1^2\mi \tfrac{1}{2} \rho_2 \g_n'(N)\big )\int_{\Hcal^2\times\Hcal^2}  \close \dtbn(U_0,\Ut_0) \pi(\d U_0,\d \Ut_0)  \nt   \\
&\qquad+\tfrac{1}{2}\rho_1^2\mi \tfrac{1}{2} \rho_2\g_n'(N) \int_{_{W_n(U_0)+W_n(\Ut_0)> N}}  \hspace*{-2cm} \dtbn(U_0,\Ut_0) \pi(\d U_0,\d \Ut_0).\label{ineq:W_dtilde:contracting:t_0:d}
\end{align}
In order to estimate the last integral on the above right--hand side, we slightly modify the proof of \cite[Lemma 4.3]{butkovsky2014subgeometric}, owing to the fact that $d(U_0,\Ut_0)=\|U_0-\Ut_0)\|_{\Hcal^1}^2$ is unbounded. To circumvent the difficulty, we note that for $n\ge 4$, (recalling $\g_n$ and $\Phi$ as in~\eqref{form:psi_n} and \eqref{form:Phi}, respectively)
\begin{align}
d(U,\Ut)\le 2\|U\|^2_{\Hcal^1}+2\|\Ut\|^2_{\Hcal^1}&\le c\Psitwo(U)+c\Psitwo(\Ut)   \nt  \\
&\le c+c(\Psitwo(U)+\Psitwo(\Ut))^{\frac{(n-1+\gamma)^2}{2n}}  \nt \\
&\le c+c( \g_n\o\g_n\o \Psitwo(U)+\g_n\o\g_n\o \Psitwo(\Ut))^{\frac{1}{2}}   \nt \\
&\le c\big(1+\beta \g_n( W_n(U)+ W_n(\Ut)) \big)^{\frac{1}{2}}. \label{ineq:d<psi(W)}
\end{align}
In the last implication above, we invoked~\eqref{cond:W_n-vs-Phi_n}. It follows from~\eqref{form:dtilde_beta} that
\begin{align} \label{cond:d-vs-psi_n(W_n)}
\dtbn(U,\Ut)\le c\big(1+\beta \g_n( W_n(U)+ W_n(\Ut)) \big)^{\frac{3}{4}}, \quad U,\Ut\in\Hcal^2,
\end{align}
holds for a positive constant $c=c(\beta,n)$ independent of $U,\Ut$. In particular,~\eqref{cond:d-vs-psi_n(W_n)} implies
\begin{align*}
\int_{_{W_n(U_0)+W_n(\Ut_0)> N}}  \hspace*{-2cm} \dtbn(U_0,\Ut_0) \pi(\d U_0,\d \Ut_0)
&\le c\int_{_{W_n(U_0)+W_n(\Ut_0)> N}}  \hspace*{-2cm}  \big(1+\beta\g_n(W_n(U_0)+ W_n(\Ut_0))\big)^{\frac{3}{4}}\pi(\d U_0,\d \Ut_0),
\end{align*}
where $c>0$ is a positive constant independent of $N$. To bound the above right--hand side, observe that
\begin{align*}
&\int_{_{W_n(U_0)+W_n(\Ut_0)> N}}  \hspace*{-2cm} \big(1+\beta \g_n(W_n(U_0)+ W_n(\Ut_0))\big)^{\frac{3}{4}}\pi(\d U_0,\d \Ut_0)\\
&\le \int_{_{W_n(U_0)+W_n(\Ut_0)> N}}  \hspace*{-2cm} 1+\beta^{\frac{3}{4}}\g_n(W_n(U_0)+ W_n(\Ut_0))^{\frac{3}{4}}\pi(\d U_0,\d \Ut_0)\\
&\le  \Big(\g_n(N)^{-\frac{3}{4}}+\beta^{\frac{3}{4}}\Big)\int_{_{W_n(U_0)+W_n(\Ut_0)> N}}  \hspace*{-2cm} \g_n(W_n(U_0)+ W_n(\Ut_0))^{\frac{3}{4}}\pi(\d U_0,\d \Ut_0)\\
&\le  \Big(K_2^{-\frac{3}{4}}+\beta^{\frac{3}{4}}\Big)\int_{_{W_n(U_0)+W_n(\Ut_0)> N}}  \hspace*{-2cm} \g_n(W_n(U_0)+ W_n(\Ut_0))^{\frac{3}{4}}\pi(\d U_0,\d \Ut_0).
\end{align*}
We invoke Holder and Markov inequalities to further deduce
\begin{align*}
&\int_{_{W_n(U_0)+W_n(\Ut_0)> N}}  \hspace*{-2cm} \g_n(W_n(U_0)+ W_n(\Ut_0))^{\frac{3}{4}}\pi(\d U_0,\d \Ut_0)\\
&\le  \g_n(N)^{-\frac{3}{4}}\int_{\Hcal^2\times\Hcal^2}  \g_n(W_n(U_0)+ W_n(\Ut_0))\pi(\d U_0,\d \Ut_0)\\
&\le \g_n(N)^{-\frac{3}{4}}\Big(\nu_1( \g_n\o W_n)+\nu_2(\g_n\o W_n)\Big).
\end{align*}
It follows that
\begin{align*}
&\int_{_{W_n(U_0)+W_n(\Ut_0)> N}}  \hspace*{-2cm} \dtbn(U_0,\Ut_0) \pi(\d U_0,\d \Ut_0)\\
&\le C(K_1,K_2,\beta,R)\g_n(N)^{-\frac{3}{4}}\big(\nu_1( \g_n\o W_n)+\nu_2(\g_n\o W_n)\big).
\end{align*}
In the above, we emphasize that $C(K_1,K_2,\beta,R)$ does not depend on $N$. Hence, provided $N$ is large enough such that
\begin{align} \label{cond:N:1}
&C(K_1,K_2,\beta,R)\g_n(N)^{-\frac{3}{4}}\big(\nu_1( \g_n\o W_n)+\nu_2(\g_n\o W_n)\big)\le \tfrac{1}{2}\W_{\dtbn}(\nu_1,\nu_2),
\end{align}
and that
\begin{align} \label{cond:N:2}
\g_n'(N)<\rho_1^2/\rho_2,
\end{align}
we have
\begin{align} \label{cond:N:3}
\W_{\dtbn}\big(P_{t_0}\nu_1,P_{t_0}\nu_2\big)
&\le \big(1-\tfrac{1}{4} \rho_2 \g_n'(N)\big )\W_{\dtbn}(\nu_1,\nu_2).
\end{align}
We note that from~\eqref{cond:d-vs-psi_n(W_n)}, 
\begin{align*}
\W_{\dtbn}(\nu_1,\nu_2) \le c\big(1+\beta \big(\nu_1( \g_n\o W_n)+\nu_2(\g_n\o W_n)\big)\big),
\end{align*}
whereas 
\begin{align*}
\nu_1( \g_n\o W_n)+\nu_2(\g_n\o W_n) \ge 2\g_n(1),
\end{align*}
thanks to the fact that $W_n\ge 1$. In other words, it holds that
\begin{align*}
\inf_{\nu_1\neq \nu_2\in \Pcal_{\g_n\o \Phi^n}(\Hcal^2) }\frac{\nu_1( \g_n\o W_n)+\nu_2(\g_n\o W_n)}{\W_{\dtbn}(\nu_1,\nu_2)}\ge c>0.
\end{align*}
As a consequence, we infer the existence of $C_n^*=C_n^*(K_2,\beta,R,\rho_1)>0$ such that $N$ given by
\begin{align*}
N:=\g_n^{-1}\Big( C_n^* \Big(\tfrac{\nu_1( \g_n\o W_n)+\nu_2(\g_n\o W_n)}{\W_{\dtbn}(\nu_1,\nu_2)}\Big)^{\frac{4}{3}} \Big),
\end{align*}
satisfies~\eqref{cond:N:1}, \eqref{cond:N:2}, \eqref{cond:N:3} as well as the requirement $N>R$. Here, we emphasieze that $C^*_n$ is independent of $\nu_1$ and $\nu_2$. Finally, it follows from~\eqref{cond:N:3} that~\eqref{ineq:W_dtilde:contracting:t_0} holds. The proof is thus complete.

\end{proof}

\begin{remark} \label{rem:d<psi(W)}
In view of the estimate \eqref{ineq:d<psi(W)}, we see that the proof of Lemma~\ref{lem:W_dtilde:contracting:t_0} is still valid for any other distance $\dt$ as long as it satisfies Proposition~\ref{prop:ergodicity} and is dominated by $\g_n(W_n)$ in the following sense 
\begin{align*}
\dt(U,\Ut)\le c\big(1+\beta\g_n(W_n(U) + W_n(\Ut) )\big)^\varepsilon,
\end{align*}
for some constants $\varepsilon\in(0,1)$ and $c>0$ independent of $U,\Ut$.
\end{remark}

Having established the contracting property of $P_{t_0}$ with respect to $\W_{\dtbn}$ in one step, we upgrade~\eqref{ineq:W_dtilde:contracting:t_0} to a polynomial decaying bound for a sequence of suitably chosen time steps.

\begin{lemma} \label{lem:W_dtilde:contracting:m_k}
Given $\nu_1,\nu_2\in \Pcal_{\g_n\o \Phi^n}(\Hcal^2)$, suppose that $\{m_k\}_{k\ge 1}$ is a sequence of strictly increasing positive integers satisfying
\begin{align} \label{cond:m_k}
P_{m_kt_0}\nu_1(\g_n\o W_n)+P_{m_kt_0}\nu_1(\g_n\o W_n)\le C(\nu_1,\nu_2), \quad k\ge 1,
\end{align}
for some positive constant $C(\nu_1,\nu_2)$ independent of $k$. Let $\dtbn$ be the distance--like function from Lemma~\ref{lem:W_dtilde:contracting:t_0}. Then, the following holds
\begin{align} \label{ineq:W_dtilde:contracting:m_k}
\W_{\dtbn}\big(P_{m_kt_0}\nu_1,P_{m_kt_0}\nu_2\big)
&\le c_nC(\nu_1,\nu_2) k^{-\frac{3(n-1+\gamma)}{4(1-\gamma)}},
\end{align}
for some positive constant $c_n$ independent of $k$, $\nu_1$ and $\nu_2$.
\end{lemma}
\begin{proof} The argument is similarly to that of \cite[Lemma 4.4]{butkovsky2014subgeometric} while making use of~\eqref{cond:d-vs-psi_n(W_n)}. 

From~\eqref{cond:d-vs-psi_n(W_n)}, observe that
\begin{align*}
\W_{\dtbn}(\nu_1,\nu_2)&\le c\Big( 1+\tfrac{3}{4}\beta\big[\nu_1(\g_n\o W_n)+\beta\nu_2(\g_n\o W_n ) \big]\Big)\\
&\le c\big[\nu_1(\g_n\o W_n)+\nu_2(\g_n\o W_n ) \big].
\end{align*}
In the last implication above, $c_*>0$ is a positive constant independent of $\nu_1,\nu_2\in \Pcal_{\g_n\o \Phi^n}(\Hcal^2)$.  As a consequence of~\eqref{cond:m_k}, we obtain
\begin{align} \label{ineq:W_dtilde:contracting:m_k:a}
\W_{\dtbn}\big(P_{m_{k}t_0}\nu_1,P_{m_kt_0}\nu_2\big) \le c_*\big[P_{m_kt_0}\nu_1(\g_n\o W_n)+P_{m_kt_0}\nu_2(\g_n\o W_n ) \big]\le c_*\cdot C(\nu_1,\nu_2).
\end{align}

Turning to~\eqref{ineq:W_dtilde:contracting:m_k}, we set
\begin{align*}
a_{k}=\frac{\W_{\dtbn}\big(P_{m_kt_0}\nu_1,P_{m_kt_0}\nu_2\big)}{c_*C(\nu_1,\nu_2)}.
\end{align*}
In view of~\eqref{ineq:W_dtilde:contracting:t_0}, it holds that
\begin{align*}
a_{k+1} &\le \Big[ 1-c_n^*\g_n'\Big(\g_n^{-1}\Big( C_n^* \Big[\frac{ P_{m_kt_0}\nu_1( \g_n\o W_n)+P_{m_kt_0}\nu_2(\g_n\o W_n)}{c_*C(\nu_1,\nu_2)}\Big]^{\frac{4}{3}} a_{k}^{-\frac{4}{3}} \Big) \Big)\Big]
a_{k}\\
&\le \Big[ 1-c_n^*\g_n'\Big(\g_n^{-1}\Big( C_n^*  a_{k}^{-\frac{4}{3}} \Big) \Big)\Big]
a_{k},
\end{align*}
where in the last implication above, we invoked~\eqref{ineq:W_dtilde:contracting:m_k:a} with the fact that $\g_n^{-1}$ is increasing whereas $\g_n'$ is decreasing. In light of \cite[Lemma 4.2]{butkovsky2014subgeometric}, we deduce
\begin{align*}
a_k\le g_n^{-1}(k),
\end{align*}
where
\begin{align*}
g_n(x)=\int_x^1\frac{1}{t \cdot c_n^*\g_n'\Big(\g_n^{-1}\Big( C_n^*  t^{-\frac{4}{3}} \Big)\Big)}\d t,\quad 0<x\le 1.
\end{align*}
By making a change of variable $y:=\g_n^{-1}( C_n^*  t^{-\frac{4}{3}} )$, we deduce (recalling $\g_n$ as in~\eqref{form:psi_n})
\begin{align*}
g_n(x)=\frac{3}{4c_n^*(1-p_n)}\Big(  \big(C_n^*x^{-\frac{4}{3}}\big)^{\frac{1-p_n}{p_n}}  -(C_n^*)^{\frac{1-p_n}{p_n}}\Big),
\end{align*}
where $p_n=\frac{n-1+\gamma}{n}$. It follows that
\begin{align*}
a_{k}\le c (k+1)^{-\frac{3p_n}{4(1-p_n)}} ,\quad k\ge 1.
\end{align*}
This establishes~\eqref{ineq:W_dtilde:contracting:m_k}, as claimed.

\end{proof}

Next, we construct a sequence $\{m_k\}$ satisfying~\eqref{cond:m_k} in such a way that they diverge to infinity not much faster than $k$.

\begin{lemma} \label{lem:m_k}
given $U\in\Hcal^2$, define $m_0=0$ and for $k\ge 1$
\begin{align} \label{form:m_k}
m_k=\inf\big\{m>m_{k-1}:P_{mt_0}(\g_n\o K_1W_n)(U)\le 2K_2+W_n(U) \big  \}.
\end{align}
Then, for all $k\ge 1$, $m_k\le 2k$. Furthermore, 
\begin{align} \label{ineq:m_k}
P_{m_kt_0}\delta_U(\g_n\o W_n) +P_{m_kt_0}\nu(\g_n\o W_n)\le  \frac{3K_2+W_n(U)}{\g_N(K_1)},
\end{align}
where $\nu$ is the unique invariant measure of $P_t$.
\end{lemma}
\begin{proof}
From~\eqref{cond:W_n}, we note that for $U\in\Hcal^2$
\begin{align*}
P_{it_0}(\g_n\o K_1 W_n)(U)\le P_{it_0}W_n(U)-P_{(i+1)t_0}W_n(U)+K_2,\quad i=0,1,2,\cdots
\end{align*}
It follows that
\begin{align} \label{ineq:m_k:a}
\sum_{i=0}^{2k} P_{it_0}(\g_n\o K_1 W_n)(U) \le W_n(U)+(2k+1)K_2.
\end{align}
Now suppose by contradiction that $m_k\ge 2k+1$. Observe that
\begin{align*}
\text{card}\big(\{1,\dots,2k\}\setminus\{m_1,\dots,m_{k-1}\}\big)\ge k+1.
\end{align*}
By the definition of $m_k$, we infer
\begin{align*}
\sum_{i=0}^{2k} P_{it_0}(\g_n\o K_1W_n)(U) \ge (k+1)(2K_2+W_n(U)),
\end{align*}
which contradicts~\eqref{ineq:m_k:a}. We therefore conclude that $m_k\le 2k$.

Turning to~\eqref{ineq:m_k}, since $\nu$ has finite $p$--moments in $\Hcal^2$ for all $p>0$ (see \eqref{ineq:moment-bound:nu}), we deduce from~\eqref{cond:W_n} while making use of invariance that
\begin{align*}
\nu(\g_n\o K_1W_n)\le K_2.
\end{align*}
In particular, this implies that
\begin{align*}
P_{m_kt_0}\nu(\g_n\o K_1W_n)\le K_2,
\end{align*}
i.e.,
\begin{align*}
P_{m_kt_0}\nu(\g_n\o W_n)\le \frac{K_2}{\g_n(K_1)}.
\end{align*}
On the other hand, we invoke the definition of $m_k$ in~\eqref{form:m_k} to see that
\begin{align*}
P_{m_kt_0}\delta_U(\g_n\o K_1 W_n) = P_{m_kt_0}(\g_n\o K_1W_n)(U)\le 2K_2+W_n(U).
\end{align*}
Hence,
\begin{align*}
P_{m_kt_0}\delta_U(\g_n\o W_n)+P_{m_kt_0}\nu(\g_n\o W_n)\le \frac{3K_2+W_n(U)}{\g_N(K_1)},
\end{align*}
thereby finishing the proof.
\end{proof}

Having obtained the required auxiliary results, we are now in a position to conclude Theorem~\ref{thm:polynomial-mixing}. See also \cite[Theorem 2.4]{butkovsky2014subgeometric}.

\begin{proof}[Proof of Theorem~\ref{thm:polynomial-mixing}]
Fixing $U\in\Hcal^2$, let $\{m_k\}_{k\ge 1}$ be the sequence defined in~\eqref{form:m_k}. In view of~\eqref{ineq:m_k}, the pair of probability measures $(\delta_U,\nu)$ satisfies condition~\eqref{cond:m_k} with $C(\delta_U,\nu)=3K_2+W_n(U)$. As a consequence, we infer from~\eqref{ineq:W_dtilde:contracting:m_k} the existence of a positive constant $c$ independent of $k$ and $U$ such that
\begin{align*}
\W_{\dtbn}\big(P_{m_kt_0}\delta_U,\nu\big)
&\le c(1+W_n(U)) (k+1)^{-\frac{3(n-1+\gamma)}{4(1-\gamma)}}\\
&\le c(1+\Phi(U)^n) (k+1)^{-\frac{3(n-1+\gamma)}{4(1-\gamma)}}.
\end{align*}
In the last implication above, we invoke the fact that $W_n$ is dominated by $\Phi^n$ (see \eqref{cond:W_n-vs-Phi_n}). Also, since $m_k\le 2k$ by virtue of Lemma~\ref{lem:m_k}, we invoke~\eqref{ineq:W_dtilde:contracting:t_0} to see that
\begin{align*}
\W_{\dtbn}\big(P_{2kt_0}\delta_U,\nu\big) \le \W_{\dtbn}\big(P_{m_kt_0}\delta_U,\nu\big)\le c(1+\Phi(U)^n) (k+1)^{-\frac{3(n-1+\gamma)}{4(1-\gamma)}},\quad k\ge 1.
\end{align*}
To deduce a bound in term of $\W_d$, we note that for $n\ge 3$,
\begin{align*}
d(U,\Ut)&\le c\Psitwo(U)+c\Psitwo(\Ut)    \\
&\le c+c(\Psitwo(U)+\Psitwo(\Ut))^{\frac{(n-1+\gamma)^2}{n}}     \\
&\le c\big(1+\beta \g_n( W_n(U)+ W_n(\Ut))\big),
\end{align*}
whence $\W_{d}(\nu_1,\nu_2)\le c\W_{\dtbn}(\nu_1,\nu_2)$. It follows that
\begin{align*}
\W_{d}\big(P_{2kt_0}\delta_U,\nu\big) \le c(1+\Phi(U)^n) (k+1)^{-\frac{3(n-1+\gamma)}{4(1-\gamma)}},\quad k\ge 1.
\end{align*}
So, for all $t\ge 2t_0$, we deduce
\begin{align*}
\W_{d}\big(P_{t}\delta_U,\nu\big) \le \W_{d}\big(P_{[\frac{t}{2t_0}]}\delta_U,\nu\big) & \le c(1+\Phi(U)^n) ([\tfrac{t}{2t_0}]+1)^{-\frac{3(n-1+\gamma)}{4(1-\gamma)}}\\
&\le  c(1+\Phi(U)^n) (t+1)^{-\frac{3(n-1+\gamma)}{4(1-\gamma)}}.
\end{align*}
This establishes~\eqref{ineq:polynomial-mixing} for $T_*=2t_0$, thereby finishing the proof.

\end{proof}

\subsection{Proof of Proposition~\ref{prop:ergodicity}}
\label{sec:ergodicity:auxiliary}

Turning to Proposition~\ref{prop:ergodicity}, we will employ Lemma~\ref{lem:moment-bound:ubar} to established the required contracting property and $d$--small sets. As an intermediate step, though, in Lemma~\ref{lem:contracting:W_d} below, we assert an estimate analogous to~\eqref{ineq:moment-bound:|ubar|:exponential-decay} in term of $\W_{d}\big( P_t(U_0,\cdot),P_t(\Ut_0,\cdot) \big) $. 
\begin{lemma} \label{lem:contracting:W_d}
 For all $t\ge0$, and $U_0,\Ut_0\in\Hcal^1$,
\begin{align}
\label{ineq:contracting:W_d(U_0,Utilde_0)<|U_0-Utilde_0|.e^(U_0+Utilde_0)}
&\W_{d}\big( P_t(U_0,\cdot),P_t(\Ut_0,\cdot) \big)   \nt \\
& \le \|U_0-\Ut_0\|_{\Hcal^1}^2\big( 1+C\varepsilon\big)\exp\Big\{\varepsilon^2C\big(\|U_0\|^2_{\Hcal^1}+\|\Ut_0)\|^2_{\Hcal^1}\big)- \tfrac{1}{8}\ve t\Big\},
\end{align}
holds for some positive constants $C$ independent of $\ve$, $t$, $U_0$ and $\Ut_0$.

\end{lemma}

For the sake of clarity, we defer the proof of Lemma~\ref{lem:contracting:W_d} to the end of this section. We now provide the proof of Proposition~\ref{prop:ergodicity}.

\begin{proof}[Proof of Proposition~\ref{prop:ergodicity}]
1. The estimate~\eqref{ineq:Lyapunov:Phi_n} is the same as~\eqref{ineq:moment-bound:H^2}, established in Lemma~\ref{lem:moment-bound:H^2}.

2. With regard to~\eqref{ineq:contracting:W_d(U_0,Utilde_0)<d(U_0,Utilde_0)}, let $U(t)$ and $\Ut(t)$ respectively be the solutions of~\eqref{eqn:wave} with initial conditions $U_0$ and $\Ut_0$. In view of the expression~\eqref{form:W_d}, we see that
\begin{align*}
\W_{d}\big( P_t(U_0,\cdot),P_t(\Ut_0,\cdot) \big) \le \E\,d\big(U(t),\Ut(t)\big) = \E\|U(t)-\Ut(t)\|_{\Hcal^1}^2 \le \|U_0-\Ut_0\|_{\Hcal^1}^2.
\end{align*}
In the last implication above, we employed~\eqref{ineq:moment-bound:|ubar|^2<|u_0|^2}. The estimate~\eqref{ineq:contracting:W_d(U_0,Utilde_0)<d(U_0,Utilde_0)}
now follows from the above inequality.

3. Turning to~\eqref{ineq:contracting:W_d(U_0,Utilde_0)<(1-rho)d(U_0,Utilde_0)}, we note that $\Psitwo$ as in~\eqref{form:Phi} satisfies 
\begin{align*}
\Psitwo(u,v)&=  \|u\|^2_{H^2}+\|v\|^2_{H^1}+\la u,v\ra_{H^1} +\tfrac{1}{2}\|u\|^2_{H^1}+\|\f(v)\|_{L^1}    \\
&\ge \tfrac{1}{2}\alpha_1\big(\|u\|^2_{H^1}+\|v\|^2_H\big),
\end{align*}
where $\alpha_1$ is the first eigenvalue of $A$ as in~\eqref{eqn:Ae_k=alpha_k.e_k}. Hence, $U\in B^R_{n}$ implies $\|U\|_{\Hcal^1}^{2n}\le 2R/\alpha_1$. 

Now, 
in view of~\eqref{ineq:contracting:W_d(U_0,Utilde_0)<|U_0-Utilde_0|.e^(U_0+Utilde_0)}, for all $\varepsilon$ sufficiently small, it holds that
\begin{align*}
&\W_{d}\big( P_t(U_0,\cdot),P_t(\Ut_0,\cdot) \big)   \nt \\
& \le \|U_0-\Ut_0\|_{\Hcal^1}^2\big( 1+C\varepsilon\big)\exp\Big\{\varepsilon^2C\big(\|U_0\|^2_{\Hcal^1}+\|\Ut_0)\|^2_{\Hcal^1}\big)- \tfrac{1}{8}\ve t\Big\}\\
&\le  \|U_0-\Ut_0\|_{\Hcal^1}^2\big( 1+C\varepsilon\big)\exp\Big\{C\varepsilon^2R^{1/n}- \tfrac{1}{8}\ve t\Big\}.
\end{align*}
In the above, we emphasize that $C>0$ is independent of $\ve,\,R,\,t$, $U_0$ and $\Ut_0$. We now pick $\varepsilon$ small and $t_0$ large enough such that
\begin{align*}
C R^{1/n}\ve <\tfrac{1}{16},\quad\text{and}\quad t_0\ge \max\{16C,1\}.
\end{align*}
It follows that for $t\ge t_0$
\begin{align*}
( 1+C\varepsilon)\exp\Big\{C\varepsilon^2R^{1/n}- \tfrac{1}{8}\ve t\Big\} \le ( 1+C\varepsilon)\exp\big\{- \tfrac{1}{16}\ve t\big\}<1,
\end{align*}
whence 
\begin{align*}
\W_{d}\big( P_t(U_0,\cdot),P_t(\Ut_0,\cdot) \big) \le (1-\rho)\|U_0-\Ut_0\|_{\Hcal^1}^2,
\end{align*}
where $1-\rho=( 1+C\varepsilon)\exp\big\{- \tfrac{1}{16}\ve t\big\}\in(0,1)$. 
This produces~\eqref{ineq:contracting:W_d(U_0,Utilde_0)<(1-rho)d(U_0,Utilde_0)}, thereby finishing the proof.
\end{proof}

Lastly, we provide the proof of Lemma~\ref{lem:contracting:W_d}, which ultimately concludes the proof of Theorem~\ref{thm:polynomial-mixing}.

\begin{proof}[Proof of Lemma~\ref{lem:contracting:W_d}]
For all $\varepsilon>0$ sufficiently small, by Sobolev embedding and Cauchy--Schwarz inequality, we note that
\begin{align}
\big(1-\tfrac{\varepsilon}{2\max\{\alpha_1,1\}}\big)\big(\|u\|^2_{H^1}+\|v\|^2_{H}\big)&\le 
\|u\|^2_{H^1}+\|v\|^2_{H}+\varepsilon\la u,v\ra_H  \nt \\
&\le \big(1+\tfrac{\varepsilon}{2\max\{\alpha_1,1\}}\big)\big(\|u\|^2_{H^1}+\|v\|^2_{H}\big), \label{ineq:(u,v)}
\end{align}
where $\alpha_1$ is the first eigenvalue of $A$ as in~\eqref{eqn:Ae_k=alpha_k.e_k}. Next, let $U(t)$ and $\Ut(t)$ respectively be the solutions of~\eqref{eqn:wave} with initial conditions $U_0$ and $\Ut_0$. Recall from~\eqref{ineq:moment-bound:|ubar|:exponential-decay} that
\begin{align*} 
&\E\Big[\|(u(t)-\ut(t),v(t)-\vt(t))\|^2_{\Hcal^1}+\varepsilon\la u(t)-\ut(t),v(t)-\vt(t)\ra_H  \Big] \\
&\le \big(\|(u_0-\ut_0,v_0-\vt_0)\|^2_{\Hcal^1}+\varepsilon\la u_0-\ut_0,v_0-\vt_0\ra_H \big) \\
&\qquad \times \big( 1+C\varepsilon^2\big)\exp\Big\{\varepsilon^2C\big(\|(u_0,v_0)\|^2_{\Hcal^1}+\|(\ut_0,\vt_0)\|^2_{\Hcal^1}\big)- \tfrac{1}{8}\ve t\Big\},\quad t\ge 0.
\end{align*}
This together with~\eqref{ineq:(u,v)} implies
\begin{align*}
&\E\|(u(t)-\ut(t),v(t)-\vt(t))\|_{\Hcal^1}^2\\
 &\le \|U_0-\Ut_0\|_{\Hcal^1}^2\big(1+\tfrac{\varepsilon}{2\max\{\alpha_1,1\}}\big)/\big(1-\tfrac{\varepsilon}{2\max\{\alpha_1,1\}}\big) \\
 &\qquad\times\big( 1+C\varepsilon^2\big)\exp\Big\{\varepsilon^2C\big(\|U_0\|^2_{\Hcal^1}+\|\Ut_0)\|^2_{\Hcal^1}\big)- \tfrac{1}{8}\ve t\Big\},
 \end{align*}
 whence (by taking $\varepsilon$ small enough)
 \begin{align}
 &\E\|(u(t)-\ut(t),v(t)-\vt(t))\|_{\Hcal^1}^2 \nt \\
 &\le \|U_0-\Ut_0\|_{\Hcal^1}^2\big( 1+C\varepsilon\big)\exp\Big\{\varepsilon^2C\big(\|U_0\|^2_{\Hcal^1}+\|\Ut_0)\|^2_{\Hcal^1}\big)- \tfrac{1}{8}\ve t\Big\}, \label{ineq:contracting:E(ubar,vbar)}
\end{align}
for some positive constant $C$ independent of $\ve$, $t$, $U_0$ and $\Ut_0$.

Turning back to~\eqref{ineq:contracting:W_d(U_0,Utilde_0)<|U_0-Utilde_0|.e^(U_0+Utilde_0)}, we invoke~\eqref{ineq:contracting:E(ubar,vbar)} together with~\eqref{form:W_d} to obtain
\begin{align*}
\W_{d}\big( P_t(U_0,\cdot),P_t(\Ut_0,\cdot) \big)  &\le \E\|(u(t)-\ut(t),v(t)-\vt(t))\|_{\Hcal^1}^2 \\
&\le  \|U_0-\Ut_0\|_{\Hcal^1}^2\big( 1+C\varepsilon\big)\exp\Big\{\varepsilon^2C\big(\|U_0\|^2_{\Hcal^1}+\|\Ut_0)\|^2_{\Hcal^1}\big)- \tfrac{1}{8}\ve t\Big\}.
\end{align*}
This establishes~\eqref{ineq:contracting:W_d(U_0,Utilde_0)<|U_0-Utilde_0|.e^(U_0+Utilde_0)}, as claimed.

\end{proof}

\section*{Acknowledgment}
The author would like to thank the anonymous reviewer for their valuable comments and suggestions.

\bibliographystyle{abbrv}
{\footnotesize\bibliography{wave-bib}}

\end{document}